\documentclass[reqno]{amsart}
\usepackage{amssymb,amsmath,latexsym,times}
\usepackage[foot]{amsaddr}
\usepackage{color}
\usepackage{graphicx}
\usepackage{array}
\usepackage{tikz}
\usepackage{hyperref}
\usetikzlibrary{decorations.markings}
\usetikzlibrary{arrows}

\newcommand{\mcB}{\mathcal{B}}
\newcommand{\mcF}{\mathcal{F}}
\newcommand{\mcG}{\mathcal{G}}

\newcommand{\mcZ}{\mathcal{Z}}
\newcommand{\from}{\colon}
\newcommand{\meet}{\wedge}
\newcommand{\join}{\vee}
\newcommand{\cl}{\operatorname{cl}}
\newcommand{\bed}{\operatorname{\Omega}}
\newcommand{\delete}{\backslash}
\newcommand{\contract}{\slash}
\newcommand{\restrictto}{|}

\newcommand{\sg}[1]{#1}

\hypersetup{
    colorlinks=false,
    pdfborder={0 0 0},
}

\newtheorem{theorem}{Theorem}[section]
\newtheorem{conjecture}[theorem]{Conjecture}
\newtheorem{proposition}[theorem]{Proposition}
\newtheorem{lemma}[theorem]{Lemma}
\newtheorem{definition}[theorem]{Definition}
\newtheorem{corollary}[theorem]{Corollary}

\tikzset{top vert/.style={circle, fill=black, inner sep=2pt, outer sep=0pt}}
\tikzset{bot vert/.style={circle, fill=black, inner sep=2pt}}
\tikzset{middlearrow/.style={
        decoration={
            markings, mark= at position 0.88 with {\arrow[scale=1.5]{stealth}},
        },
        thick,
        postaction={decorate},
    }
}
\tikzset{middlearrow_g/.style={
        color={black!40},
        decoration={
            markings, mark= at position 0.88 with {\arrow[scale=1.5]{stealth}},
        },line width = 0.95pt,thick,
draw=black,
        postaction={decorate},
    }
}
\tikzset{middlearrow_d/.style={
        decoration={
            markings, mark= at position 0.88 with {\arrow[scale=1.5]{stealth}},
        },
        postaction={decorate},
    }
}
\tikzset{middlearrow_f/.style={
        decoration={
            markings, mark= at position 0.88 with {\arrow[scale=1.5]{stealth}},
        },line width = 0.95pt,thick,
        postaction={decorate},
    }
}
\tikzset{middlearrow_var/.style={
        decoration={
            markings, mark= at position 1 with {\arrow[scale=1.5]{stealth}},
        },
        thick,
        postaction={decorate},
    }
}
     
\begin{document}

\title{An infinite family of excluded minors for strong
  base-orderability} 
\author{Joseph E.\ Bonin} \author{Thomas J.\ Savitsky}
\address{Department of Mathematics\\
  The George Washington University\\
  Washington, DC 20052} \email{jbonin@gwu.edu}
\email{savitsky@gwmail.gwu.edu}
\begin{abstract}
  We discuss a conjecture of Ingleton on excluded minors for
  base-orderability, and, extending a result he stated, we prove that
  infinitely many of the matroids that he identified are excluded
  minors for base-orderability, as well as for the class of gammoids.
  We prove that a paving matroid is base-orderable if and only if it
  has no $M(K_4)$-minor.  For each $k\ge 2$, we define the property of
  $k$-base-orderability, which lies strictly between base-orderability
  and strong base-orderability, and we show that $k$-base-orderable
  matroids form what Ingleton called a complete class.  By
  generalizing an example of Ingleton, we construct a set of matroids,
  each of which is an excluded minor for $k$-base-orderability, but is
  $(k-1)$-base-orderable; the union of these sets, over all $k$, is an
  infinite set of base-orderable excluded minors for strong
  base-orderability.
\end{abstract}

\maketitle

\section{Introduction}\label{sec:intro}

Basis-exchange properties are of long-standing interest in matroid
theory (see Kung's survey \cite{white_book1_ch4}).  Condition (BE) in
the following definition of a matroid is a simple basis-exchange
property: a matroid $M$ is an ordered pair $(\mcB, E(M))$ where $E(M)$
is a finite set and $\mcB$ is a non-empty collection of subsets of
$E(M)$ (the bases) such that
\begin{enumerate}
\item[(BE)] 
   \emph{if $B_1, B_2 \in \mcB$ and $x\in B_1-B_2$, then there
    is some $y\in B_2-B_1$ so that \\ $(B_1 - \sg{x}) \cup {y} \in \mcB$.}
\end{enumerate}
Brualdi \cite{brualdi_dep} showed that property (BE) is equivalent to
the following, seemingly stronger, \emph{symmetric basis-exchange
  property}:
\begin{quote}
  \emph{if $B_1,B_2\in \mcB$ and $x \in B_1-B_2$, then there is some
    $y \in B_2-B_1$ so that both $(B_1 - \sg{x}) \cup y$ and $(B_2 -
    \sg{y}) \cup \sg{x}$ are in $\mcB$.}
\end{quote}
Brylawski \cite{MR0337634}, Greene \cite{MR0311494}, and Woodall
\cite{woodall_ex} showed that the \emph{multiple symmetric exchange
  property} holds for all matroids:
\begin{quote}
  \emph{if $B_1,B_2\in \mcB$ and $X \subseteq B_1 - B_2$, then there
    is some $Y \subseteq B_2 - B_1$ so that both $(B_1 - X) \cup Y$
    and $(B_2 - Y) \cup X$ are in $\mcB$.}
\end{quote}
In \cite{brualdi_dep}, Brualdi also showed that the \emph{bijective
  exchange property} holds for all matroids:
\begin{quote}
  \emph{if $B_1,B_2\in \mcB$, then there is a bijection $\sigma \from
    B_1 \to B_2$ so that, for all $x \in B_1$, the set $(B_1 - \sg{x})
    \cup \sg{\sigma(x)}$ is in $\mcB$.}
\end{quote}

In contrast, our work here is motivated by the following
basis-exchange properties that are not possessed by all matroids.

\begin{definition}
  A matroid $M$ is \emph{base-orderable} if, given any two bases $B_1$
  and $B_2$, there is a bijection $\sigma \from B_1 \to B_2$ such that
  for every $x \in B_1$, both $(B_1 - \sg{x}) \cup \sg{\sigma(x)}$ and
  $(B_2 - \sg{\sigma(x)}) \cup \sg{x}$ are bases.

  A matroid $M$ is \emph{strongly base-orderable} if, given any two
  bases $B_1$ and $B_2$, there is a bijection $\sigma \from B_1 \to
  B_2$ such that for every $X \subseteq B_1$,
    \begin{enumerate}
    \item[\emph{(*)}] $(B_1 - X) \cup \sigma(X)$ is a basis, and
    \item[\emph{(**)}] $(B_2 - \sigma(X)) \cup X$ is a basis.
    \end{enumerate}
\end{definition}

To the best of our knowledge, the notion of base-orderability first
appeared in both \cite{brualdi_scrimger} and \cite{brualdi_dep} at
about the same time.  Brualdi and Scrimger \cite{brualdi_scrimger}
showed that all transversal matroids are strongly base-orderable (and
hence base-orderable).  The property of base-orderability appeared
(without the term) in Brualdi \cite{brualdi_dep} as a natural
strengthening of the basis-exchange properties discussed there.

Not all matroids are base-orderable; in particular, the cycle matroid
$M(K_4)$ is not.  We denote the class of base-orderable matroids by
$\mathcal{BO}$ and that of strongly base-orderable matroids by
$\mathcal{SBO}$.  (In this paper, by a \emph{class} of matroids we
mean a set of matroids that is closed under isomorphism.)  Clearly,
$\mathcal{SBO} \subseteq \mathcal{BO}$.  Ingleton
\cite{ingleton_nonbo} gave an example that shows that this containment
is proper.  In Section~\ref{sec:em_sbo} we generalize his example; we
construct an infinite collection of excluded minors for strong
base-orderability, each of which is base-orderable.

It is easy to show that the class of base-orderable matroids is
minor-closed, but describing its excluded minors remains an open
problem.  In Section~\ref{sec:ingletons_conjecture}, we discuss a
conjecture of Ingleton on the excluded minors.  Much of our work arose
by exploring ideas in Ingleton's paper \cite{ingleton_nonbo}, to which
we owe a great debt.  A number of our results and constructions grew
from seeds in that paper, which, while providing a wealth of
intriguing ideas, contains few proofs.  To give a more complete
account of this topic, we also offer proofs of some of the assertions
that Ingleton made, either without proof or with a minimal sketch of
the proof.  In Section~\ref{sec:bedigraph}, we lay the groundwork for
Section~\ref{sec:ingletons_conjecture} and also prove that a paving
matroid is base-orderable if and only if it has no $M(K_4)$-minor.  In
Section~\ref{sec:cyclic_flats}, we review cyclic flats, which we use
extensively thereafter.  In Section~\ref{sec:em_gammoids}, we prove a
special case of Ingleton's conjecture.  There we describe an infinite
family of excluded minors for the class $\mathcal{BO}$, each of which
has precisely six cyclic flats; these matroids are also excluded
minors for the class of gammoids.  (Recall that a \emph{gammoid} is a
minor of a transversal matroid.)

Ingleton defined complete classes of matroids in
\cite{ingleton_related}.

\begin{definition}\label{def:complete}
  A class of matroids is \emph{complete} if it is closed under the
  operations of minors, duals, direct sums, truncations, and induction
  by directed graphs.
\end{definition}

It is known that $\mathcal{BO}$ and $\mathcal{SBO}$ are complete
classes.  Ingleton \cite{ingleton_related} noted that the class of
gammoids is complete and that each non-empty complete class contains
all gammoids.  In particular, $\mathcal{SBO}$ contains all gammoids.
This containment is proper since, for instance, the V\'amos matroid is
strongly base-orderable, but it is not a gammoid since gammoids are
representable over the real numbers.  These three complete classes
form part of a hierarchy that Ingleton, again in
\cite{ingleton_related}, introduced.  We quote:
\begin{quote}
  ``There is scope for introducing an infinity of complete classes
  between $\mathcal{BO}$ and $\mathcal{SBO}$ by appropriate
  limitations on the cardinals of subsets $X$ for which (*) is to hold,
  but these do not seem to have been studied.''
\end{quote}
In Section \ref{sec:prelim}, we begin to study some of these classes;
we introduce the class $k$-$\mathcal{BO}$ of $k$-base-orderable
matroids, where $k$ is a fixed positive integer, and address the first
three operations in Definition \ref{def:complete}.  The last two
operations are treated in Section \ref{sec:kbo}, using a reformulation
of completeness that we prove in Section~\ref{sec:complete}.

We assume familiarity with basic matroid theory.  For notation, we
follow Oxley \cite{oxleybook}.  We use $2^{S}$ to denote the set of
subsets of a set $S$.  For a family $\mcF$ of sets, we shorten
$\cap_{X\in \mcF} X$ to $\cap \mcF$ and $\cup_{X\in \mcF} X$ to $\cup
\mcF$.  For a graph $G$ with vertex set $V$, the \emph{neighborhood}
of $X \subseteq V$, denoted $N_G(X)$, is
$$N_G(X) = \{v \in V : xv \text{ is an edge of } G \text{ for some }
x \in X\}.$$ If $G$ is clear from context, we may omit the subscript.
A \emph{digraph} is a directed graph.

\section{$k$-base-orderable matroids}\label{sec:prelim}

In this section, we begin to explore some variations on the concept of
base-orderability that fit the mold described by Ingleton in the quote
above.  While our main results about the matroids we introduce below,
$k$-base-orderable matroids, are in Sections \ref{sec:kbo} and
\ref{sec:em_sbo}, in this section we treat some properties that we use
freely throughout the rest of the paper.

\begin{definition}\label{def:kbo}
  Fix a positive integer $k$.  A \emph{$k$-exchange-ordering} for a
  pair of bases $B_1$ and $B_2$ of a matroid $M$ is a bijection
  $\sigma \from B_1 \to B_2$ such that, for every subset $X$ of $B_1$
  with $|X| \le k$, both $(B_1 - X) \cup \sigma(X)$ and $(B_2-
  \sigma(X))\cup X$ are bases of $M$.

  A matroid is \emph{$k$-base-orderable} if each pair of its bases has
  a $k$-exchange ordering.
\end{definition}

For a fixed $k \ge 1$, we denote the class of $k$-base-orderable
matroids by $k$-$\mathcal{BO}$.  Thus,
$1$-$\mathcal{BO}=\mathcal{BO}$.  We usually call a
$1$-exchange-ordering an \emph{exchange-ordering}.  A matroid $M$ is
strongly base-orderable if and only if it is $r(M)$-base-orderable.

\begin{proposition}\label{prop:kbobasics}
  Let $k$ and $l$ be positive integers.
  \begin{enumerate}
  \item If $M$ is $\lceil r(M)/2 \rceil$-base-orderable, then it is
    strongly-base-orderable.
  \item A bijection $\sigma \from B_1 \to B_2$ is a
    $k$-exchange-ordering if and only if $\sigma^{-1} \from B_2 \to
    B_1$ is also a $k$-exchange-ordering.
  \item If $\sigma \from B_1 \to B_2$ is a $k$-exchange-ordering, then
    \begin{enumerate}
    \item $\sigma(x) = x$ for every $x \in B_1 \cap B_2$, and\
    \item $\sigma$ is also an $l$-exchange-ordering for every integer
      $l$ with $1 \le l \le k$.
    \end{enumerate}
  \item If $M$ and $N$ are $k$-base-orderable, then so is $M\oplus N$.
  \item If $M$ is $k$-base-orderable, then so is $M^{*}$, as well as
    $M \delete x$ and $M \contract x$ for each $x \in E(M)$.
  \end{enumerate}
\end{proposition}

\begin{proof}
  Parts (1)--(4) are immediate.  For part (5), let $B_1^{*}$ and
  $B_2^{*}$ be bases of $M^{*}$.  There is a $k$-exchange-ordering
  $\sigma \from E(M) - B_2^{*} \to E(M) - B_1^{*}$ for $M$. Define
  $\sigma^{*} \from B_1^{*} \to B_2^{*}$ by
  \begin{equation*} \sigma^{*}(x) =
    \begin{cases} x, & \text{if $x\in B_1^{*} \cap B_2^{*}$,}\\
      \sigma(x), & \text{if $x \in B_1^{*} - B_2^{*}$.}
    \end{cases}
  \end{equation*} 
  It is easy to check that $\sigma^{*}$ is a $k$-exchange-ordering for
  $M^{*}$.

  It now suffices to treat $M\delete x$ since $M \contract x = (M^{*}
  \delete x)^{*}$.  Let $B_1$ and $B_2$ be bases of $M \delete x$.  If
  $x$ is not a coloop, then a $k$-exchange ordering for the bases
  $B_1$ and $B_2$ of $M$ serves as such for $M \delete x$.  Otherwise,
  $M$ has a $k$-exchange-ordering $\sigma \from B_1 \cup x \to B_2
  \cup x$, and since $\sigma(x) = x$, its restriction $\sigma|_{B_1}
  \from B_1 \to B_2$ is a $k$-exchange-ordering for $M \delete x$.
\end{proof}

The next two lemmas will be useful when discussing excluded minors.

\begin{lemma}\label{lem:kbo_disjoint_union}
  If $M$ is not $k$-base-orderable, then it has a minor $N$ whose
  ground set is the union of two disjoint bases of $N$ that have no
  $k$-exchange ordering.
\end{lemma}

\begin{proof}
  Let $A$ and $B$ be bases of $M$ that have no $k$-exchange-ordering.
  Take
  \begin{equation*}
    N = M \contract (A\cap B) \delete (E(M) - (A\cup B)). \qedhere
  \end{equation*}
\end{proof}

\begin{lemma}\label{lem:kbodorc}
  If a rank-$r$ matroid $M$ with $|E(M)|= 2r$ is not in
  $k$-$\mathcal{BO}$ but either
  \begin{enumerate}
  \item all single-element contractions $M \contract x$ are in
    $k$-$\mathcal{BO}$, or
  \item all single-element deletions $M \delete x$ are in
    $k$-$\mathcal{BO}$,
  \end{enumerate}
  then $M$ is an excluded minor for $k$-$\mathcal{BO}$.  The same is
  true if we replace $k$-$\mathcal{BO}$ by $\mathcal{SBO}$.
\end{lemma}

\begin{proof}
  By duality, it suffices to treat the case in which condition (1)
  holds.  Since $M$ is not in $k$-$\mathcal{BO}$ but all of its
  single-element contractions are, $M$ has no coloops.  Fix $y \in
  E(M)$.  Let $B_1$ and $B_2$ be bases of $M \delete y$.  Since $M
  \delete y$ has $2r-1$ elements and rank $r$, the bases $B_1$ and
  $B_2$ cannot be disjoint.  Fix $x\in B_1 \cap B_2$.  In $M \contract
  x$, there is a $k$-exchange-ordering $\sigma \from B_1 - \sg{x} \to
  B_2 - \sg{x}$ by condition (1).  Extending $\sigma$ by setting
  $\sigma(x) = x$ gives a $k$-exchange-ordering for $B_1$ and $B_2$ in
  $M \delete y$.  Thus, $M$ is an excluded minor for
  $k$-$\mathcal{BO}$.
\end{proof}

\section{The basis-exchange digraph}\label{sec:bedigraph}

The following construction is often helpful when examining
basis-exchange properties.

\begin{definition}\label{def:bedigraph}
  Let $A$ and $B$ be bases of $M$.  The \emph{basis-exchange digraph
    of $A$ and $B$ with respect to $M$} is the directed bipartite
  graph $\bed^M_{A,B}$ on $2\,r(M)$ vertices with bipartition $\{A,
  B\}$ (using disjoint copies of $A$ and $B$ if $A \cap B \ne
  \varnothing$) where, for $a\in A$ and $b\in B$,
  \begin{enumerate}
  \item $(a, b)\in E(\bed^M_{A,B})$ if and only if $(B - \sg{b}) \cup
    \sg{a}$ is not a basis of $M$, and
  \item $(b, a)\in E(\bed^M_{A,B})$ if and only if $(A - \sg{a}) \cup
    \sg{b}$ is not a basis of $M$.
  \end{enumerate}
\end{definition}

We shorten $\bed^M_{A,B}$ to $\bed_{A,B}$ when $M$ is clear from the
context.  Figure \ref{fig:Delta3} illustrates the definition.

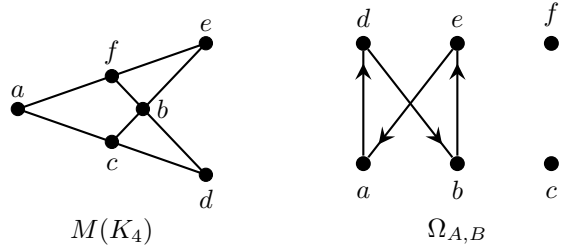
\begin{figure}[h]
  \centering
 \begin{tikzpicture}[scale=1.25]
  \draw[thick](0,0.5)--(2,-0.2);
  \draw[thick](0,0.5)--(2,1.2);
  \draw[thick](1,0.15)--(2,1.2);
  \draw[thick](1,0.85)--(2,-0.2);
  \filldraw (1,0.85) node[above] {$f$\rule[-3pt]{0pt}{1pt}} circle  (2pt);
  \filldraw (2,1.2) node[above] {$e$\rule[-1pt]{0pt}{1pt}} circle  (2pt);
  \filldraw (2,-0.2) node[below] {$d$\rule{0pt}{9pt}} circle  (2pt);
  \filldraw (1,0.15) node[below] {$c$\rule{0pt}{7pt}} circle  (2pt);
  \filldraw (1.33,0.5) node[right] {$\,b$} circle  (2pt);
  \filldraw (0,0.5) node[above] {$a$\rule[-1pt]{0pt}{1pt}} circle  (2pt);
  \node at (1,-0.8) {$M(K_4)$};
  \end{tikzpicture}
  \hspace{1.5cm}
 \begin{tikzpicture}[]
    \node[bot vert, label=below:{$a$\rule{0pt}{6.5pt}}](a) at (0, 0)
    {}; \node[bot vert, 
    label=below:$b$](d) at (1.25, 0) {}; \node[bot vert,
    label=below:{$c$\rule{0pt}{6.5pt}}](b) at (2.5, 0) {};

    \node[top vert, label=above:$d$](c) at (0, 1.6) {}; \node[top vert,
    label=above:$e$](f) at (1.25, 1.6) {}; \node[top vert,
    label=above:$f$](e) at (2.5, 1.6) {};

    \draw[middlearrow](a) -- (c); \draw[middlearrow](c) -- (d);
    \draw[middlearrow](d) -- (f); \draw[middlearrow](f) -- (a);

  \node at (1.25,-0.9) {$\bed_{A,B}$};
  \end{tikzpicture}
  \caption{The cycle matroid $M(K_4)$ and its basis-exchange digraph
    for the bases $A = \{a,b,c\}$ and $B=\{d,e,f\}$.}
  \label{fig:Delta3}
\end{figure}

While some authors use this term for a different graph, our definition
is consistent with the critical graphs that Ingleton defined in
\cite{ingleton_nonbo} (see Definition \ref{def:ingleton_critical}
below).  The following proposition is well-known and easy to prove.

\begin{proposition}\label{prop:kbo_fund_circ}
  Let $A$ and $B$ be bases of a matroid $M$.  For $a \in A - B$, the
  fundamental circuit of $a$ with respect to $B$, denoted by $C(a,B)$,
  is $$ \{a\} \cup \{b \in B : (a, b) \text{ is not an edge of }
  \bed_{A,B}\} .$$
\end{proposition}

We now recall Hall's Theorem on matchings in a bipartite graph.  It
was originally stated and proved for systems of distinct
representatives by Philip Hall in \cite{halls_thm}.

\begin{theorem}\label{thm:Hall}
  Let $G$ be a bipartite graph with bipartition $\{X, Y\}$.  There is
  a matching that covers $X$ if and only if $|X'| \le |N(X')|$ for all
  sets $X' \subseteq X$.
\end{theorem}

The next lemma, from \cite{ingleton_nonbo}, is of crucial importance,
so we fill in the sketch of the proof that was given there.

\begin{lemma}\label{lem:nbo}
  Let $A$ and $B$ be bases of a matroid $M$.  There is no
  exchange-ordering between $A$ and $B$ if and only if some subgraph
  of $\bed_{A,B}$ is an orientation of a complete bipartite graph
  $K_{s,t}$ for some $s\geq 2$ and $t\geq 2$ with $s + t = r(M) + 1$.
\end{lemma}

\begin{proof}
  Let $\bar{\Omega}$ be the undirected bipartite graph with the same
  bipartition as $\bed_{A,B}$, in which $ab$, with $a\in A$ and $b\in
  B$, is an edge if and only if neither $(a,b)$ nor $(b,a)$ is an edge
  of $\bed_{A,B}$.  In other words, $ab$ is in $E(\bar{\Omega})$
  exactly when both $(A - \sg{a}) \cup \sg{b}$ and $(B - \sg{b}) \cup
  \sg{a}$ are bases of $M$.  Thus, $A$ and $B$ have an
  exchange-ordering if and only if $\bar{\Omega}$ has a perfect
  matching.

  By Hall's Theorem, $\bar{\Omega}$ has no perfect matching if and
  only if there is a subset $X \subseteq A$ with $|X| > |N(X)|$, where
  $N(X) \subseteq B$ is the neighborhood of $X$ in $\bar{\Omega}$.
  Now $$|X| + |B - N(X)| = r(M) + |X| - |N(X)|,$$ so the inequality
  $|X| > |N(X)|$ is equivalent to $|X| + |B - N(X)| \ge r(M) + 1$.
  Also, for every $x\in X$ and $y \in B - N(X)$, either $(x,y)$ or
  $(y,x)$ is an edge of $\bed_{A,B}$.  It follows that $\bar{\Omega}$
  has no perfect matching if and only if $\bed_{A,B}$ has a
  restriction that is an orientation of $K_{s,t}$ for some $s$ and $t$
  with $s + t = r(M) + 1$.  By the symmetric basis-exchange property,
  neither $s$ nor $t$ can be $r$, so $s\geq 2$ and $t\geq 2$.
\end{proof}

Recall that a matroid $M$ is \emph{paving} if it contains no circuit
of size less than $r(M)$; it is \emph{sparse-paving} if both $M$ and
$M^*$ are paving.  It is well-known that the classes of paving
matroids and sparse-paving matroids are minor-closed.  From Figure
\ref{fig:Delta3} and Lemma \ref{lem:nbo}, we see that $M(K_4)$ is not
base-orderable.  In \cite{desousa_welsh}, de Sousa and Welsh proved
that a binary matroid is base-orderable if and only if it has no
$M(K_4)$-minor.  We next prove that $M(K_4)$ is also the only obstacle
to base-orderability for paving matroids.

\begin{theorem}
  A paving matroid $M$ is base-orderable if and only if $M$ has no
  $M(K_4)$-minor.
\end{theorem}

\begin{proof}
  If $M$ has an $M(K_4)$-minor, then $M$ is not base-orderable since
  $\mathcal{BO}$ is minor-closed.  We now show that if $M$ is not
  base-orderable, then it has an $M(K_4)$-minor. By
  Lemma~\ref{lem:kbo_disjoint_union}, $M$ has a minor $N$ whose ground
  set is the union of two disjoint bases of $N$, say $A$ and $B$, that
  have no exchange-ordering.  By Lemma~\ref{lem:nbo}, the
  basis-exchange digraph $\bed^N_{A,B}$ has a subgraph that is an
  orientation of $K_{s,t}$ where $s + t = r(N) + 1$.  Since $N$ is
  paving and $A \cap B = \varnothing$,
  Proposition~\ref{prop:kbo_fund_circ} implies that the out-degree of
  any vertex in $\bed^N_{A,B}$ is at most one.  Therefore
  $|V(K_{s,t})| \ge |E(K_{s,t})|$, i.e., $s + t \ge st$.  The only
  solution to this inequality with $s, t > 1$ is $s=t=2$, so $r(N) =
  s+t-1=3$, and so $|E(N)| = 6$.  Transversal matroids are
  base-orderable, and the only rank-$3$ matroid on six elements that
  is not transversal is $M(K_4)$, so $N$ is $M(K_4)$.
\end{proof}

This theorem is interesting in light of the recent work of Pendavingh
and van der Pol \cite{counting_minor_closed} that the number of
sparse-paving matroids with no $M(K_4)$-minor is asymptotic to the
best-known lower bound on the number of sparse-paving matroids.  It is
conjectured that asymptotically almost all matroids are sparse-paving,
so it seems reasonable to also conjecture that almost all matroids are
base-orderable.

The next two results are implicit in Ingleton \cite{ingleton_nonbo}.

\begin{proposition}\label{prop:Delta_reverse}
  If $A$ and $B$ are disjoint bases of a matroid $M$ with $E(M) = A
  \cup B$, then $\bed^{M^{*}}_{A,B}$ is obtained from $\bed^M_{A,B}$
  by reversing the orientation of each edge.
\end{proposition}

\begin{proof}
  Let $a \in A$ and $b \in B$.  Then $(A - \sg{a}) \cup \sg{b}$ is a
  basis of $M$ (or $M^*$) if and only if $(B - \sg{b}) \cup \sg{a}$ is
  a basis of $M^*$ (or $M$).
\end{proof}

The next proposition limits the structure of basis-exchange digraphs
of excluded minors for $\mathcal{BO}$.  

\begin{proposition}\label{prop:no_source_sink}
  Assume that bases $A$ and $B$ of a matroid $M$ have no
  exchange-ordering.  Let $\Gamma$ be a subgraph of $\bed_{A,B}$ that
  is an orientation of $K_{s,t}$ with $s+t=r(M)+1$ and $s,t\geq 2$.
  If $\Gamma$ has either a source or a sink, then $M$ is not an
  excluded minor for base-orderability.
\end{proposition}

\begin{proof}
  By Lemma~\ref{lem:kbo_disjoint_union}, there is nothing to show
  unless $A \cap B = \varnothing$ and $E(M) = A \cup B$.  Let $A =
  \{a_1, a_2, \dots, a_r\}$ and $B = \{b_1, b_2, \dots, b_r\}$, with
  $\{a_1, a_2, \dots, a_s\} \cup \{b_1, b_2, \dots, b_t\}$ being the
  vertex set of $\Gamma$.  By Propositions \ref{prop:kbobasics} and
  \ref{prop:Delta_reverse}, it suffices to treat the case in which
  $\Gamma$ has a source, say $a_1$.  By the symmetric basis-exchange
  property, we may assume that $B' = (B - \sg{b_r}) \cup \sg{a_1}$ is
  a basis of $M$.  We claim that $\bed_{A,B'}$ has a subgraph that is
  an orientation of $K_{s,t}$.  Clearly if $(b_j, a_i) \in
  E(\bed_{A,B})$, then $(b_j, a_i) \in E(\bed_{A, B'})$ as well.  Let
  $(a_i, b_j)$, with $i>1$, be an edge of $\Gamma$.  Now $(a_1, b_j)
  \in E(\bed_{A,B})$ since $a_1$ is a source of $\Gamma$.
  Proposition~\ref{prop:kbo_fund_circ} gives $\{a_1, a_i\} \subseteq
  \cl(B-b_j)$.  Therefore
  $$r\big((B - b_j) \cup \{a_1, a_i\}\big) = r(M) - 1.$$ Thus,
  $(B - \{b_j, b_r\}) \cup \{a_1, a_i\}$ is not a basis of $M$, so
  $(a_i, b_j)\in E(\bed_{A, B'})$.  Lastly note that $(a_1, b_j) \in
  E(\bed_{A,B'})$ since $a_1$ is itself a member of $B'$.  Thus,
  $\bed_{A,B'}$ has a subgraph that is an orientation of $K_{s,t}$, so
  $M \contract a_1 \delete b_r$ is not base-orderable by
  Lemma~\ref{lem:nbo}.
\end{proof}

\section{Background on cyclic flats of matroids}\label{sec:cyclic_flats}

The rest of this paper makes heavy use of cyclic flats, which we
briefly review in this section.  For a fuller account, see
\cite{bonin_cyclic_flats}.

Let $M$ be a matroid with rank function $r$.  A set $X \subseteq E(M)$
is \emph{cyclic} if $X$ is a (possibly empty) union of circuits;
equivalently, $X$ is cyclic if the restriction $M\restrictto{X}$ has
no coloops.  The collection of cyclic flats of $M$, denoted $\mcZ(M)$,
is a lattice under set-inclusion, with the same join as in the lattice
of flats, namely, $X\vee Y = \cl(X\cup Y)$.  An attractive feature of
cyclic flats is that they are well-behaved under duality.

\begin{proposition}\label{prop:cfdual}
  For a matroid $M$, we have $\mcZ(M^{*}) = \{E(M)-X : X \in
  \mcZ(M)\}$.
\end{proposition}

A matroid $M$ is determined by $E(M)$ along with the pairs $(A,r(A))$
for $A\in \mcZ(M)$.  The following result from \cite{JulieSimsThesis,
  bonin_cyclic_flats} formulates matroids in these terms.

\begin{theorem}\label{thm:cyclic_flats}
  Let $\mcZ$ be a set of subsets of a set $S$ and let $r$ be an
  integer-valued function on $\mcZ$.  There is a matroid $M$ with
  $S=E(M)$ for which $\mcZ=\mcZ(M)$ and $r(X) =r_M(X)$ for all $X\in
  \mcZ$ if and only if
  \begin{enumerate}
  \item[(Z0)] $\mcZ$ is a lattice under inclusion,
  \item[(Z1)] $r(0_\mcZ) = 0$, where $0_\mcZ$ is the least element of
    $\mcZ$,
  \item[(Z2)] $0 < r(Y) - r(X) < |Y - X|$ for all sets $X, Y$ in
    $\mathcal{Z}$ with $X \subsetneq Y$, and
  \item[(Z3)] for all incomparable sets $X, Y$ in $\mcZ$ (i.e.,
    neither contains the other),
    $$r(X) + r(Y) \ge r(X \join Y) + r(X \meet Y) + |(X \cap Y) - (X
    \meet Y)|.$$
  \end{enumerate}
\end{theorem}

The rank of a set $X\subseteq E(M)$, in terms of the ranks of cyclic
flats, is given by
\begin{equation}\label{eq:cfrank}
  r_M(X) = \min\{r(A) + |X-A| : A \in \mcZ(M)\}.
\end{equation}

We also require information about the cyclic flats of minors.

\begin{lemma}\label{lem:cfminor}
  Let $M$ be a matroid, and let $x \in E(M)$.  If $F\in \mcZ(M \delete
  x)$, then either $F$ or $F\cup x$ is a cyclic flat of $M$.  The same
  conclusion holds if $F\in \mcZ(M \contract x)$.
\end{lemma}

\begin{proof}
  If a cyclic flat $F$ of $M \delete x$ is also a flat of $M$, then
  $F\in \mcZ(M)$; otherwise $x$ is not a coloop of $M\restrictto{F\cup
    x}$ and $\cl_M(F) = F\cup x$, so $F\cup x\in \mcZ(M)$. The second
  assertion follows by duality.
\end{proof}

We say that a matroid $N$ is \emph{freer} than $M$ if $E(M)=E(N)$ and
$r_M(X) \le r_N(X)$ for all $X\subseteq E(M)$.  We next formulate this
order (the weak order) in terms of cyclic flats.

\begin{lemma}\label{lem:cfwo}
  A matroid $N$ is freer than $M$ if and only if for all $F\in
  \mcZ(N)$, there is some $A\in \mcZ(M)$ with $r_M(A)+|F-A|\leq
  r_N(F)$.
\end{lemma}

\begin{proof}
  The necessity of the condition is clear.  We focus on the converse.
  For $X \subseteq E(M)$, we have $r_N(X)=r_N(F)+|X-F|$ for some $F\in
  \mcZ(N)$.  Now $r_M(A)+|F-A|\leq r_N(F)$ for some $A\in \mcZ(M)$ by
  assumption.  Since $|X-A|\leq |X-F|+|F-A|$, we have
  \begin{equation*}
    \begin{split}
      r_M(X) & \leq r_M(A)+|X-A|\\
      & \leq r_M(A)+|F-A|+|X-F|\\
      & \leq r_N(F)+|X-F|\\
      & = r_N(X).\\
    \end{split}
  \end{equation*}
  The first and last terms are the required inequality.
\end{proof}

We will use the Mason-Ingleton characterization of transversal
matroids.  

\begin{theorem}[The Mason-Ingleton condition]
  A matroid $M$ is transversal if and only if for all nonempty
  antichains $\mathcal{F}$ of $\mathcal{Z}(M)$,
  \begin{equation}\label{eq:micond}
    r(\cap\mathcal{F})
    \le \sum\limits_{\mathcal{F'} \subseteq \mathcal{F}}
    (-1)^{|\mathcal{F'}|+1}r(\cup\mathcal{F}').
  \end{equation}
\end{theorem}

For a proof of this theorem, see \cite{bonin_fund}.  Inequality
(\ref{eq:micond}) trivially holds when $|\mathcal{F}| = 1$, and it
reduces to submodularity when $|\mathcal{F}| = 2$.  We will use the
following corollary.

\begin{corollary}\label{cor:micond}
  Let $M$ be a matroid. Fix $\mcG \subseteq 2^{E(M)}$ with $\mcZ(M)
  \subseteq \mcG$.  If inequality~\emph{(\ref{eq:micond})} holds for
  all nonempty antichains $\mathcal{F}$ of $\mathcal{G}$ with $|\mcF|
  \ge 3$, then $M$ is transversal.
\end{corollary}

\section{Ingleton's conjecture}\label{sec:ingletons_conjecture}

In \cite{ingleton_nonbo}, Ingleton discussed an infinite set of
matroids that he conjectured to be excluded minors for $\mathcal{BO}$.
His conjectured excluded minors are associated to what he called
critical graphs; however, he gave the definition of these matroids
only for critical graphs that lack a structure that we call an
obstruction.  For a critical graph with no obstructions, he gave two
families of sets and said that the bases of the associated matroid are
their common transversals; in Sections \ref{subsec:cgptm} and
\ref{subsec:noobs}, we develop a view of these matroids in terms of
cyclic flats and show that Ingleton's description of the bases applies
precisely when obstructions are absent.  The properties we prove about
obstructions show that they are relatively well-behaved, and in
Section \ref{subsec:wobs} we define a likely candidate for the
conjectured excluded minors that are associated to critical graphs
with obstructions (that material is not used in the rest of the
paper). In Section \ref{subsec:data}, we present data that supports
the conjecture about excluded minors.

\subsection{Critical graphs and pairs of transversal
  matroids}\label{subsec:cgptm}

We start with a fundamental definition due to Ingleton.

\begin{definition}\label{def:ingleton_critical}
  Let $A$ and $B$ be disjoint sets of size $r$, where $r \ge 3$.  A
  bipartite digraph $\Delta$ with bipartition $\{A, B\}$ is a
  \emph{critical graph} if there are subsets $X$ of $A$ and $Y$ of $B$
  such that
  \begin{enumerate}
  \item $|X| + |Y| = r + 1$,
  \item for all $x\in X$ and $y\in Y$, exactly one of $(x,y)$ and
    $(y,x)$ is in $E(\Delta)$,
  \item if $(u,v)\in E(\Delta)$, then $\{u,v\} \subseteq X\cup Y$, and
  \item no element of $X\cup Y$ is a source or sink of $\Delta$.
  \end{enumerate}
\end{definition}

Thus, a critical graph on $2r$ vertices is an orientation of
$K_{s,t}$, for some $s\geq 2$ and $t\geq 2$ with $s+t=r+1$, having
neither a source nor a sink, with $r-1$ isolated vertices adjoined.
For example, the digraph in Figure \ref{fig:Delta3} is a critical
graph.  Definition \ref{def:ingleton_critical} is motivated largely by
Lemma \ref{lem:nbo} and Proposition \ref{prop:no_source_sink}.  

Ingleton said that for each critical graph $\Delta$, he could
construct a matroid $M(\Delta)$ on $A\cup B$ in which $A$ and $B$ are
bases and the basis-exchange digraph $\bed^{M(\Delta)}_{A,B}$ is
$\Delta$; furthermore, all excluded minors for $\mathcal{BO}$ occur
among what he called the good specializations of these matroids
$M(\Delta)$.  (One property of good specializations is that they can
have more dependent sets.)  Thus, the idea is to construct, for each
critical graph $\Delta$, a matroid $M(\Delta)$ that has $\Delta$ as a
basis-exchange digraph (so $M(\Delta) \not \in \mathcal{BO}$ by Lemma
\ref{lem:nbo}) and whose dependent sets are, as much as possible, just
those that are forced by $\Delta$.

To see what structure $\Delta$ imposes on $M(\Delta)$, let $M$ be a
rank-$r$ matroid, with $r\ge 3$, where $E(M)$ is the disjoint union of
two bases, $A$ and $B$, and $\bed^M_{A,B}$ is a critical graph
$\Delta$.  Let $X$ and $Y$ be as in
Definition~\ref{def:ingleton_critical}.  Using fundamental circuits,
as in Proposition \ref{prop:kbo_fund_circ}, we recast what $\Delta$
gives us. We have
\begin{enumerate}
\item proper subsets $X$ of $A$ and $Y$ of $B$ with $|X|+|Y| = r+1$,
\item a fundamental circuit $C(y,A)$, for each $y\in Y$, with
  $A-X\subsetneq C(y,A)-y\subsetneq A$ and $$A= \bigcup_{y\in
    Y}C(y,A)-y,$$
\item a fundamental circuit $C(x,B)$, for each $x\in X$, with
  $B-Y\subsetneq C(x,B)-x\subsetneq B$ and $$B= \bigcup_{x\in
    X}C(x,B)-x,$$
\item whenever $x\in X$ and $y\in Y$, exactly one of the statements
  $x\in C(y,A)$ and $y\in C(x,B)$ holds, and
\item $C(b,A) = A\cup b$ for $b\in B-Y$, and $C(a,B) = B\cup a$ for
  $a\in A-X$.
\end{enumerate}
For a subset $A'$ of the basis $A$, the flat $\cl_M(A') = A' \cup
\{b\in B : C(b,A)-\sg{b} \subseteq A'\}$ has rank $|A'|$.  This flat
is cyclic if for each $a \in A'$, there is a $b\in \cl_M(A') \cap B$
with $a \in C(b, A)$.  The counterparts of these conclusions for
subsets $B'$ of $B$ also hold.  There may, of course, be circuits of
$M$ besides the fundamental circuits that $\Delta$ gives.

As we will see, in many cases the minimal structure that $\Delta$
imposes on $M(\Delta)$ is enough to determine $M(\Delta)$.  Let
$\Delta$ be a critical graph with $r$, $A$, $B$, $X$, and $Y$ as in
Definition~\ref{def:ingleton_critical}.  We begin to describe a
candidate for $M(\Delta)$ by specifying some of its cyclic flats and
their ranks.  From the observations above, we see that in order to
have $\Delta = \bed^{M(\Delta)}_{A,B}$, certain cyclic flats must be
present in $\mcZ(M(\Delta))$.  For $b\in B$, we define
\begin{equation}\label{eq:CDeltabA}
  C_\Delta(b,A) = \{b\}\cup \{a\in A : (b,a) \not\in E(\Delta)\}.
\end{equation}
We extend this notation as follows: for $B' \subseteq B$, we define
\begin{equation}\label{eq:CDeltaB'}
  C_\Delta(B',A)= \bigcup_{b\in B'}C_\Delta(b,A).
\end{equation}
Now we define $\mcZ_A$ and the ranks of its sets as follows: for each
$B' \subseteq B$, we adjoin the set
\begin{equation}\label{form:Acf}
  D_\Delta(B') = C_\Delta(B',A) \cup \{b\in B : C_\Delta(b,A) - \sg{b}
  \subseteq C_\Delta(B',A)\} 
\end{equation}
to $\mcZ_A$ and set $r(D_\Delta(B'))=|D_\Delta(B') \cap A|$.  Note
that $\varnothing$ is in $\mcZ_A$ with rank $0$ (take
$B'=\varnothing$), and $A\cup B$ is in $\mcZ_A$ with rank $r$ (take
$B'=B$, say).  Also note that
\begin{equation*}
  D_\Delta(B') = C_\Delta(D_\Delta(B')\cap B,A)
\end{equation*}
and thus
\begin{equation}\label{form:Acfimp}
  D_\Delta(B') \cap Y = \{y\in Y : C_\Delta(y,A) - \sg{y} \subseteq
  D_\Delta(B')\}. 
\end{equation}

Likewise construct $\mcZ_B$.  Specifically, for $a\in A$, we define
$$C_\Delta(a,B) = \{a\}\cup \{b\in B : (a,b) \not\in E(\Delta)\}.$$
For $A' \subseteq A$, we define
\begin{equation}\label{eq:CDeltaA'}
  C_\Delta(A',B) = \bigcup_{a\in A'}C_\Delta(a,B),
\end{equation}
and we adjoin the set
\begin{equation}\label{form:Bcf}
  D_\Delta(A') = C_\Delta(A',B) \cup \{a\in A : C_\Delta(a,B) - \sg{a}
  \subseteq C_\Delta(A',B)\} 
\end{equation}
to $\mcZ_B$ with rank $|D_\Delta(A') \cap B|$.  

Set $\mcZ_\Delta=\mcZ_A\cup \mcZ_B$.  Since $\mcZ_A\cap \mcZ_B =
\{\varnothing, A\cup B\}$, there is no ambiguity as to the ranks of
the sets in $\mcZ_\Delta$.

In the proof of the next result, we use the following observations
about the transversal matroid $M$ that arises from a bipartite graph.
By Theorem \ref{thm:Hall}, the circuits of $M$ are the subsets $W$ of
$E(M)$ for which $|N(W)|<|W|$ while $|N(Z)|\geq|Z|$ whenever
$Z\subsetneq W$.  Thus, if $W$ is a circuit and $w\in W$, then
$|N(W)|=|N(W-\{w\})|=r(W)$, and so $N(W)= N(W-\{w\})$.  Also, if
$|N(U)| = r(U)$, then $\cl(U) = \{x\,:\, N(x)\subseteq N(U)\}$.

\begin{proposition}\label{prop:fund_trans_delta}
  The set $\mcZ_A$, with the rank of each set in $\mcZ_A$ as given
  above, is the set of cyclic flats, along with their ranks, of a
  transversal matroid on $A\cup B$, and likewise for $\mcZ_B$.  The
  sets $A$ and $B$ are bases of both of these transversal matroids.
\end{proposition}

\begin{proof}
  By symmetry, it suffices to treat the assertions about $\mcZ_A$.
  For $a\in A$, let $$S_a = \{a\}\cup \{b\in B \,:\,(b,a)\not\in
  E(\Delta)\}.$$ Let $\Gamma$ be the bipartite graph with bipartition
  $\{A\cup B,\{S_a\,:\,a\in A\}\}$ and with edge set $\{x\, S_a\,:\,
  x\in S_a\}$.  Let $M$ be the transversal matroid on $A\cup B$
  defined by $\Gamma$.  It is easy to see that $A$ is a basis of $M$,
  that $|N_\Gamma(A')| = r(A')$ for all subsets $A'$ of $A$, and that, for
  each $b\in B$, the set $C_\Delta(b,A)$ in equation
  (\ref{eq:CDeltabA}) is the fundamental circuit $C_M(b,A)$.  From
  these conclusions, equation (\ref{form:Acf}), and the observations
  above, it follows that all sets in $\mcZ_A$ are in $\mcZ(M)$.

  To show that each set $Z$ in $\mcZ(M)$ is in $\mcZ_A$ and that
  $Z\cap A$ is a basis of $M|Z$, we start with a circuit $W$ of $M$.
  As noted above, $|N_\Gamma(W)| = r(W)$.  Thus, for $b\in W\cap B$,
  we have $C_M(b,A)\subseteq\cl_M(W)$.  Therefore $\cl_M(W)\cap A$ is
  a basis of $M|\cl_M(W)$.  For any $Z\in \mcZ(M)$, there are circuits
  $W_1,W_2,\ldots,W_t$ of $M$ with $Z = W_1 \cup W_2 \cup \cdots \cup
  W_t$, so, since $\cl_M(W_i)\cap A$ is a basis of $M|\cl_M(W_i)$ for
  each $i$, and $\cl_M(W_i)\subseteq Z$, it follows that $Z\cap A$ is
  a basis of $M|Z$, and, furthermore, $Z=D_\Delta(Z\cap B)$.

  Finally, to show that $B$ is a basis of $M$, assume, instead, that
  $B$ contains a circuit $W$.  We use $X$, $Y$, and $r$ as in
  Definition \ref{def:ingleton_critical}.  Since
  $|N_\Gamma(W)|=r(W)<r$, no element of $B-Y$ is in $W$, so
  $W\subseteq Y$.  Now $N_\Gamma(W) = \{S_a\,:\,a\in(A-X)\cup X'\}$
  for some subset $X'$ of $X$.  Since $W$ is a circuit, $
  |W|-1=|N_\Gamma(W)|$, so $|W|-1 = |A-X|+|X'|$.  Adding $|X-X'|$ to
  both sides gives $|X-X'|+|W|-1 = |A|=r$, so $|X-X'|+|W| = r+1$.
  Since $|X|+|Y| = r+1$, we get $W=Y$ and $X' = \varnothing$, but
  $W=Y$ gives the contradiction $|N_\Gamma(W)| = r$.  (Having
  $X'=\varnothing$ also gives a contradiction: each vertex in $W$
  would be a source of $\Delta$.)  Thus, $B$ is indeed a basis of $M$.
\end{proof}

\subsection{When $\mcZ_\Delta$ is the lattice of cyclic flats of
  $M(\Delta)$}\label{subsec:noobs}

In this section, we show that the matroid $M(\Delta)$ associated to a
critical graph $\Delta$ can have $\mcZ_\Delta$ as its lattice of
cyclic flats if and only if $\Delta$ does not contain a structure that
we call an obstruction.

We start with some examples in which $\mcZ(M(\Delta)) = \mcZ_\Delta$.
The digraph in Figure~\ref{fig:Delta3} is a critical graph,
$\Delta_3$, where $A = \{a,b,c\}$ and $B = \{d,e,f\}$; the associated
matroid $M(\Delta_3)$ is $M(K_4)$.  For a more complex example, let
$\Delta_5$ be the digraph in Figure~\ref{fig:Delta5}, where $A =
\{a_1, \dots, a_5\}$ and $B = \{b_1, \dots, b_5\}$.  Then
$M(\Delta_5)$ has the lattice of cyclic flats shown in
Figure~\ref{fig:ZMDelta_5}.

\begin{figure}[h]
  \centering
  \begin{tikzpicture}[]
    \node[top vert, label=above:$a_1$](a1) at (0, 2) {}; \node[top
    vert, label=above:$a_2$](a2) at (1.25, 2) {}; \node[top vert,
    label=above:$a_3$](a3) at (2.5, 2) {}; \node[top vert,
    label=above:$a_4$](a4) at (3.5, 2) {}; \node[top vert,
    label=above:$a_5$](a5) at (4.5, 2) {};

    \node[bot vert, label=below:$b_1$](b1) at (0, 0) {}; \node[bot
    vert, label=below:$b_2$](b2) at (1.25, 0) {}; \node[bot vert,
    label=below:$b_3$](b3) at (2.5, 0) {}; \node[bot vert,
    label=below:$b_4$](b4) at (3.5, 0) {}; \node[bot vert,
    label=below:$b_5$](b5) at (4.5, 0) {};

    \draw[middlearrow](b2) -- (a1); \draw[middlearrow](b1) -- (a2);
    \draw[middlearrow](b3) -- (a2); \draw[middlearrow](b3) -- (a2);
    \draw[middlearrow](b3) -- (a3);

    \draw[middlearrow](a1) -- (b1); \draw[middlearrow](a1) -- (b3);
    \draw[middlearrow](a2) -- (b2); \draw[middlearrow](a3) -- (b1);
    \draw[middlearrow](a3) -- (b2);
  \end{tikzpicture}
  \caption[The critical graph $\Delta_5$.]
  {The critical graph $\Delta_5$.}
  \label{fig:Delta5}
\end{figure}
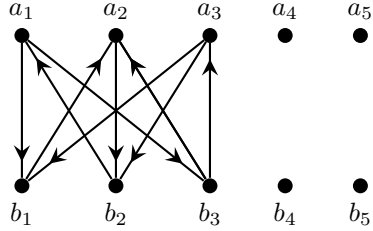

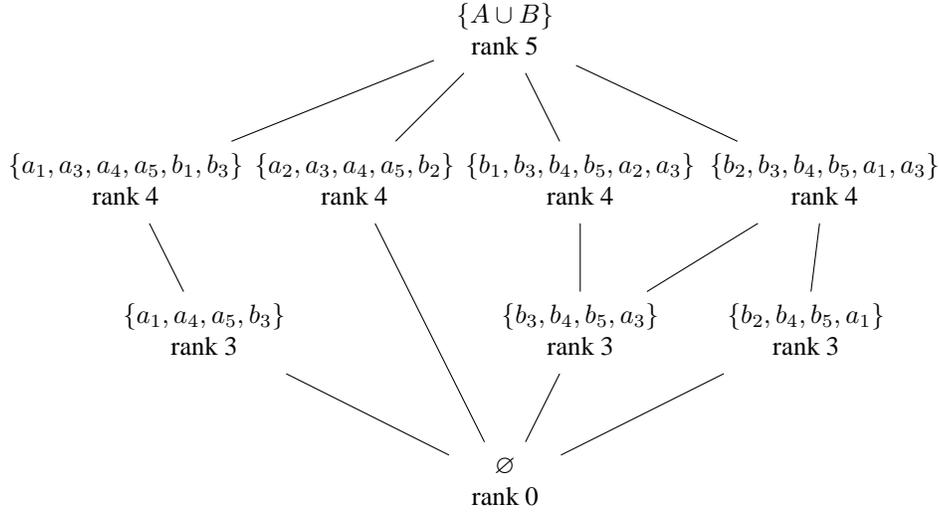
\begin{figure}[h]
\centering
\begin{tikzpicture}
    \node(S)  at (0,8) {
        \begin{tabular}{c}
            $\{A\cup B\}$\\
            rank 5\\
        \end{tabular}
    };

    \node[align=center](a1) at (-5, 6) {
        \begin{tabular}{c}
            $\{a_1, a_3, a_4, a_5, b_1, b_3\}$\\
            rank 4\\
        \end{tabular}
    };
    \node(a2) at (-2, 6) {
        \begin{tabular}{c}
            $\{a_2, a_3, a_4, a_5, b_2\}$\\
            rank 4\\
        \end{tabular}
    };
    \node(a3) at (-4, 4) {
        \begin{tabular}{c}
            $\{a_1, a_4, a_5, b_3\}$\\
            rank 3\\
        \end{tabular}
    };
    \node(b1) at (1, 6) {
        \begin{tabular}{c}
            $\{b_1, b_3, b_4, b_5, a_2, a_3\}$\\
            rank 4\\
        \end{tabular}
    };
    \node(b2) at (4.25, 6) {
        \begin{tabular}{c}
            $\{b_2, b_3, b_4, b_5, a_1, a_3\}$\\
            rank 4\\
        \end{tabular}
    };
    \node(b3) at (4, 4) {
        \begin{tabular}{c}
            $\{b_2, b_4, b_5, a_1\}$\\
            rank 3\\
        \end{tabular}
    };
    \node(b4) at (1, 4) {
        \begin{tabular}{c}
            $\{b_3, b_4, b_5, a_3\}$\\
            rank 3\\
        \end{tabular}
    };

    \node(nullset) at (0, 2) {
        \begin{tabular}{c}
            $\varnothing$\\
            rank 0\\
        \end{tabular}
    };

    \draw(S) -- (a1);
    \draw(S) -- (a2);
    \draw(a1) -- (a3);
    \draw(a2) -- (nullset);
    \draw(a3) -- (nullset);
    \draw(S) -- (b1);
    \draw(S) -- (b2);
    \draw(b1) -- (b4);
    \draw(b2) -- (b3);
    \draw(b2) -- (b4);
    \draw(b3) -- (nullset);
    \draw(b4) -- (nullset);
\end{tikzpicture}
\caption{The lattice of cyclic flats of $M(\Delta_5)$.}
\label{fig:ZMDelta_5}
\end{figure}

It follows from equation (\ref{eq:cfrank}) that if $\mcZ(M(\Delta)) =
\mcZ_\Delta$, then the bases of $M(\Delta)$ are those that are common
to the two transversal matroids in Proposition
\ref{prop:fund_trans_delta}.  In this case, $M(\Delta)$ is definitely
the matroid that Ingleton intended since he said, in
\cite{ingleton_nonbo}, ``For a large class of $\Delta$ it is possible
to define the bases of $M(\Delta)$ as the common transversals of two
families of sets,'' and he then gave the set system $\{S_a\,:\,a\in
A\}$ that we used in the proof of Proposition
\ref{prop:fund_trans_delta}, and its counterpart for $\mcZ_B$.  In
particular, $A$ and $B$ are both bases of $M(\Delta)$.  Also,
$\bed^{M(\Delta)}_{A,B}=\Delta$.  Having $\mcZ(M(\Delta)) =
\mcZ_\Delta$ also implies that $M(\Delta)$ is freest among the
matroids $N$ on $A\cup B$ in which $A$ and $B$ are bases and
$\bed^{N}_{A,B}=\Delta$; to see this, take $A=F$ in
Lemma~\ref{lem:cfwo}.

Ingleton identified the structures in the following definition.

\begin{definition}\label{def:obstruction}
  Given a critical graph $\Delta$, a pair $(K, L)$ is an
  \emph{obstruction} if
  \begin{enumerate}
  \item $\varnothing \subsetneq K \subsetneq X$ and $\varnothing
    \subsetneq L \subsetneq Y$,
  \item $(k,y)\in E(\Delta)$ for every $k\in K$ and $y\in Y-L$, and
  \item $(l,x)\in E(\Delta)$ for every $l\in L$ and $x\in X-K$.
  \end{enumerate}
\end{definition}

The inclusions $C_\Delta(K,B)-K\subseteq L\cup (B-Y)$ and
$C_\Delta(L,A)-L\subseteq K\cup (A-X)$ are equivalent to conditions
(2) and (3), respectively.

Fortunately, as the next four results show, obstructions are rather
well-behaved.  The proof of the following lemma is immediate from the
definition.

\begin{lemma}\label{lemma:obs_reverse}
  Let $\Delta$ be a critical graph, and let $\Delta'$ be the digraph
  obtained by reversing the orientation of every edge of $\Delta$.
  The pair $(K,L)$ is an obstruction of $\Delta$ if and only if $(X-K,
  Y-L)$ is an obstruction of $\Delta'$.
\end{lemma}

The next result shows that obstructions form a lattice.

\begin{proposition}\label{prop:noobslattice}
  If $(K_1,L_1)$ and $(K_2,L_2)$ are obstructions of a critical graph
  $\Delta$, then both $(K_1 \cap K_2, L_1 \cap L_2)$ and $(K_1 \cup
  K_2, L_1 \cup L_2)$ are obstructions.
\end{proposition}

\begin{proof}
  First observe that if $k \in K_1-K_2$ and $l \in L_2-L_1$, then both
  $(k,l)$ and $(l,k)$ would be edges of $\Delta$, which is
  impossible. Thus, at least one of $K_1-K_2$ and $L_2-L_1$ is empty,
  and likewise for the pair $K_2-K_1$ and $L_1-L_2$.  That is, (i)
  either $K_1\subseteq K_2$ or $L_2\subseteq L_1$ and (ii) either
  $K_2\subseteq K_1$ or $L_1\subseteq L_2$.  If $K_1\ne K_2$ and
  $L_1\ne L_2$, then conclusions (i) and (ii) imply that either (a)
  $K_1\subsetneq K_2$ and $L_1\subsetneq L_2$, or (b) $K_2\subsetneq
  K_1$ and $L_2\subsetneq L_1$; in these cases, the conclusion of the
  proposition is immediate.

  By symmetry, we may now assume that $K_1 = K_2$.  Thus, $(k,y)\in
  E(\Delta)$ for each $k\in K_1$ and $y\in Y-(L_1\cap L_2)$; likewise,
  $(l,x)\in E(\Delta)$ for each $l\in L_1\cup L_2$ and $x\in X-K_1$.
  It cannot be that $L_1 \cap L_2 = \varnothing$, for then each $k\in
  K_1$ would be a source of $\Delta$.  We also cannot have $L_1 \cup
  L_2 = Y$, since then each $x\in X-K_1$ would be a sink of
  $\Delta$. Thus, both $(K_1, L_1 \cap L_2)$ and $(K_1 , L_1 \cup
  L_2)$ are obstructions, as needed.
\end{proof}

Thus, if a critical graph has an obstruction, then it has a minimum
obstruction and a maximum obstruction.  

The next result shows that there are no obstructions if $r < 7$.
However, obstructions can and do occur if $r \ge 7$.
Figure~\ref{fig:Delta7} shows $\Delta_7$, the smallest critical
digraph, up to isomorphism, that has an obstruction.

\begin{proposition}\label{prop:obsge2}
  If $(K,L)$ is an obstruction of a critical graph $\Delta$, then each
  of the sets $K$, $L$, $X-K$, and $Y-L$ has at least two elements.
\end{proposition}

\begin{proof}
  Assume for a contradiction that $Y-L = \{y\}$.  Since $K \subsetneq
  X$ and $y$ is not a sink, there is some $x \in X-K$ with $(y, x) \in
  E(\Delta)$.  This implies that $x$ is then a sink, contrary to
  property (4) of $\Delta$.  Thus $|Y-L| \ge 2$.  By symmetry $|X-K|
  \ge 2$.  Now Lemma~\ref{lemma:obs_reverse} implies that $|K|, |L|
  \ge 2$.
\end{proof}

\begin{figure}[h]
  \centering
  \begin{tikzpicture}[]
    \node[top vert, label=above:$a_1$](a1) at (0, 2) {}; \node[top
    vert, label=above:$a_2$](a2) at (1.25, 2) {}; \node[top vert,
    label=above:$a_3$](a3) at (2.5, 2) {}; \node[top vert,
    label=above:$a_4$](a4) at (3.75, 2) {}; \node[top vert,
    label=above:$a_5$](a5) at (4.75, 2) {}; \node[top vert,
    label=above:$a_6$](a6) at (5.75, 2) {}; \node[top vert,
    label=above:$a_7$](a7) at (6.75, 2) {};

    \node[bot vert, label=below:$b_1$](b1) at (0, 0) {}; \node[bot
    vert, label=below:$b_2$](b2) at (1.25, 0) {}; \node[bot vert,
    label=below:$b_3$](b3) at (2.5, 0) {}; \node[bot vert,
    label=below:$b_4$](b4) at (3.75, 0) {}; \node[bot vert,
    label=below:$b_5$](b5) at (4.75, 0) {}; \node[bot vert,
    label=below:$b_6$](b6) at (5.75, 0) {}; \node[bot vert,
    label=below:$b_7$](b7) at (6.75, 0) {};

    \draw[middlearrow_f](a1) -- (b1); \draw[middlearrow_f](a2) --
    (b2);

    \draw[middlearrow_f](b1) -- (a2); \draw[middlearrow_f](b2) -- (a1);

    \draw[middlearrow_f](a3) -- (b3); \draw[middlearrow_f](a4) --
    (b4);

    \draw[middlearrow_f](b3) -- (a4); \draw[middlearrow_f](b4) -- (a3);

    \draw[middlearrow_g](a3) -- (b1); \draw[middlearrow_g](a3) --
    (b2); \draw[middlearrow_g](a4) -- (b1); \draw[middlearrow_g](a4)
    -- (b2);

    \draw[middlearrow_g](b3) -- (a1); \draw[middlearrow_g](b3) -- (a2);
    \draw[middlearrow_g](b4) -- (a1); \draw[middlearrow_g](b4) -- (a2);
  \end{tikzpicture}
  \caption[The critical graph, $\Delta_7$.]{The pair $(\{a_3, a_4\},
    \{b_3, b_4\})$ is an obstruction in this critical graph,
    $\Delta_7$.  The edges with gray arrows show that conditions (2)
    and (3) in Definition \ref{def:obstruction} hold.}
  \label{fig:Delta7}
\end{figure}
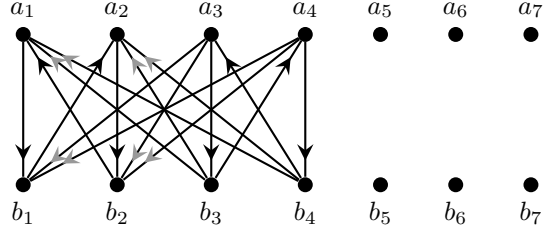

\begin{lemma}\label{lemma:mineq}
  If $(K,L)$ is the minimum obstruction of $\Delta$, then
  $C_\Delta(L,A)-L= K\cup (A-X)$ and $C_\Delta(K,B)-K= L\cup (B-Y)$.
\end{lemma}

\begin{proof}
  For $k \in K$, if $C_\Delta(L,A)-L\subseteq (K-k)\cup (A-X)$, then
  $(K-k,L)$ would also be an obstruction, contrary to $(K,L)$ being
  the minimum.  The second equality follows by symmetry.
\end{proof}

We now treat a key result.

\begin{proposition}\label{prop:noobs}
  Let $\Delta$ be a critical graph.  The collection $\mcZ_\Delta$,
  with the ranks given before Proposition \ref{prop:fund_trans_delta},
  satisfies conditions \emph{(Z0)--(Z3)} in Theorem
  \ref{thm:cyclic_flats} (and so defines a matroid) if and only if
  $\Delta$ has no obstructions.
\end{proposition}

\begin{proof}
  From equation (\ref{form:Acf}), each $I \in \mcZ_A - \{A\cup B,
  \varnothing\}$ is a proper superset of $A-X$ that is disjoint from
  $B-Y$, while from equation (\ref{form:Bcf}), each $J \in \mcZ_B -
  \{A\cup B, \varnothing\}$ is a proper superset of $B-Y$ that is
  disjoint from $A-X$; thus, $I \meet J = \varnothing$ and $I \join J
  = A \cup B$.  Also, if $I,J \in \mcZ_A$, then $I \join J$ and $I
  \meet J$ are as in $\mcZ_A$, and likewise if $I,J \in \mcZ_B$.
  Thus, condition (Z0) holds.

  Each of the other conditions follows from its counterpart in
  $\mcZ_A$ or $\mcZ_B$ with one exception: we must check whether
  condition (Z3) holds for all $I, J$ with $I \in \mcZ_A - \{A\cup B,
  \varnothing\}$ and $J \in \mcZ_B - \{A\cup B, \varnothing\}$.
  Condition (Z3) for such an $I$ and $J$ is equivalent to
  \begin{equation}\label{eq:mdeltaZ3}
    r + |I\cap J| \le |I\cap A| + |J\cap B|,
  \end{equation}
  which, since $|A-X|+|B-Y|=r-1$, is equivalent to $1 + |I\cap J| \le
  |I\cap X| + |J\cap Y|$.  Assume that this inequality fails, that
  is, $$|I\cap J| \geq |I\cap X| + |J\cap Y|.$$ Since $I\cap J$ is the
  disjoint union of $I\cap J\cap X$ and $I\cap J\cap Y$, the last
  inequality gives both $I\cap J \cap X = I\cap X$ and $I\cap J\cap Y
  = J\cap Y$, from which we get $I\cap X \subseteq J\cap X$ and $J\cap
  Y \subseteq I\cap Y$.  Equation (\ref{form:Acfimp}) and the
  inclusion $J\cap Y \subseteq I\cap Y$ give $$C_\Delta(J\cap
  Y,A)-(J\cap Y) \subseteq I\cap A = (I\cap X)\cup (A-X).$$ Likewise,
  $C_\Delta(I\cap X,B)-(I\cap X) \subseteq (J\cap Y)\cup (B-Y)$
  follows from $I\cap X \subseteq J\cap X$, so $(I\cap X, J\cap Y)$ is
  an obstruction.  Thus, if condition (Z3) fails, then $\Delta$ has an
  obstruction.

  Now assume that $\Delta$ has an obstruction.  Let $(K, L)$ be the
  minimum obstruction.  Set $I=D_\Delta(L)$, so $I \in \mcZ_A$, and
  set $J=D_\Delta(K),$ so $J\in \mcZ_B$.  By Lemma \ref{lemma:mineq},
  $$r(I) + r(J) = |A-X|+|K|+|B-Y|+|L| = r-1+|K|+|L|.$$
  On the other hand, $I\meet J = \varnothing$ and $K\cup L\subseteq
  I\cap J$, so
  $$r(I \join J) + r(I \meet J) + |(I \cap J) - (I \meet J)|
  \geq r + |K| + |L|.$$ Thus condition (Z3) of
  Theorem~\ref{thm:cyclic_flats} fails.
\end{proof}

Thus, when $\Delta$ has no obstructions, we take $M(\Delta)$ to be the
matroid with lattice of cyclic flats equal to $\mcZ_\Delta$.  We next
reformulate a conjecture that Ingleton made in \cite{ingleton_nonbo}.
We prove a special case in Theorem~\ref{thm:k4like}.

\begin{conjecture}\label{conj:noobs}
  If $\Delta$ is a critical graph with no obstructions, then
  $M(\Delta)$ is an excluded minor for $\mathcal{BO}$.
\end{conjecture}

\subsection{A candidate for $M(\Delta)$ when $\Delta$ has
  obstructions}\label{subsec:wobs}
The only thing that Ingleton said in \cite{ingleton_nonbo} about
$M(\Delta)$ when $\Delta$ has an obstruction is that ``the set of
bases of $M(\Delta)$ has to be a suitably chosen proper subset of the
set of common transversals'' of $\{S_a\,:\,a\in A\}$ and its
counterpart for $\mcZ_B$.  Thus, we cannot be sure that what we
present below is what he had in mind.  As we note below, our candidate
for $M(\Delta)$ has a property that Ingleton asserted for the matroids
he had in mind.  Also, the computational evidence cited in Section
\ref{subsec:data} lines up with Ingleton's conjecture.  The material
in this section is not used in the rest of the paper.
  
Proposition~\ref{prop:noobslattice} justifies the following notation.
When $\Delta$ has an obstruction, let $P$ be the set $K_0 \cup L_0$
where $(K_0, L_0)$ is the minimum obstruction of $\Delta$, and let $Q$
be the set $K_1 \cup L_1 \cup (A-X) \cup (B-Y)$ where $(K_1, L_1)$ is
the maximum obstruction of $\Delta$.

\begin{proposition}\label{prop:obslattice}
  Let $\Delta$ be a critical graph having an obstruction.  Set $r(P) =
  |P|-1$ and $r(Q)=r-1$.  Each of the following collections of sets,
  with the ranks defined above, satisfies the conditions of
  Theorem~\ref{thm:cyclic_flats} and so defines a matroid:
  $$\mcZ_\Delta^P=\mcZ_\Delta \cup \{P\}, \qquad
  \mcZ_\Delta^Q=\mcZ_\Delta \cup \{Q\}, \qquad\text{and}\qquad
  \mcZ_\Delta^{P,Q}=\mcZ_\Delta \cup \{P, Q\}.$$
\end{proposition}

\begin{proof}
  Each of the sets $P\cap X$, $P\cap Y$, $A-Q$, and $B-Q$ has at least
  two elements by Proposition~\ref{prop:obsge2}.  Recall that
  $C_\Delta$ is given by equations (\ref{eq:CDeltaB'}) and
  (\ref{eq:CDeltaA'}), and $D_\Delta$ by (\ref{form:Acf}) and
  (\ref{form:Bcf}).

  We first treat $\mcZ_\Delta^P$.  We begin with the lattice
  structure.  Consider $I,J \in \mcZ_\Delta^P - \{A\cup B,
  \varnothing\}$.  By symmetry, we may take $I \in \mcZ_A$.

  Assume first that $J \in \mcZ_A$. Clearly $I \join J$ is the same as
  in $\mcZ_A$, as is $I \meet J$ if $P \not\subseteq I\cap J$.  Sets
  in $\mcZ_A$ that contain $P$ also contain $D_\Delta(L_0)$, so if $P
  \subseteq I\cap J$, then $D_\Delta(L_0)\subseteq I$ and
  $D_\Delta(L_0)\subseteq J$, so $I \meet J$ is again the same as in
  $\mcZ_A$.

  If $J \in \mcZ_B$, then $I \join J = A\cup B$ and
  \begin{equation*}
    I \meet J = 
    \begin{cases}
      P & \text{if} \ P \subseteq I \cap J,\\
      \varnothing & \text{otherwise.}
    \end{cases}
  \end{equation*}

  Comparable sets trivially have a meet and a join, so we may assume
  that the remaining sets to treat, $I$ and $P$, are incomparable, in
  which case it is easy to see that $I \meet P = \varnothing$ and
  \begin{equation}\label{eq:joinIP}
    I \join P = D_\Delta(L_0\cup (I\cap B)).
  \end{equation}
  Thus, condition (Z0) holds.  Note that $r(I\join P) = |(I\cap A)\cup
  K_0|$; we will use this below.

  Next, we check condition (Z3) for $P$ and an incomparable set $I\in
  \mcZ_A$.  The following statements are equivalent:
  \begin{enumerate}
  \item $r(I\join P) + r(I\meet P) + |(I\cap P) - (I\meet P)| \le r(I)
    + r(P)$,
  \item $|(I\cap A)\cup K_0| + |I\cap K_0| + |I\cap L_0| \le |I\cap A|
    + |K_0|+|L_0|-1$,
  \item $|K_0-I| + |I\cap K_0| + |I\cap L_0| \le |K_0|+|L_0|-1$,
  \item $|I\cap L_0| \le |L_0| - 1,$ and
  \item $I\cap L_0 \subsetneq L_0$.
  \end{enumerate}
  Assume statement (5) fails, so $L_0\subseteq I$.  Lemma
  \ref{lemma:mineq} gives $C_\Delta(L_0,A)-L_0= K_0\cup (A-X)$.  The
  inclusion $L_0\subseteq I$ gives $C_\Delta(L_0,A)-L_0\subseteq I\cap
  A$.  Thus, $K_0\subseteq I$, so $P\subseteq I$, contrary to the
  assumption that they are incomparable.

  We next check condition (Z3) for $I \in \mcZ_A$ and $J \in \mcZ_B$.
  If the inequality in condition (Z3) holds for $I$ and $J$ in the
  lattice $\mcZ_\Delta$, then its counterpart clearly holds in the
  lattice $\mcZ_\Delta^P$.  Thus, assume the inequality fails for $I$
  and $J$ in $\mcZ_\Delta$.  The proof of Proposition~\ref{prop:noobs}
  shows that $P\subseteq I\cap J$, so $I \meet J = P$.  Therefore the
  inequality in condition (Z3) amounts to
  $$|I\cap A| + |J\cap B| \geq r+ |P|-1 + | (I\cap J) - P|,$$ which is
  equivalent to $|I\cap A| + |J\cap B| \geq r+ |I\cap J|-1$, and so to
  $$|I\cap X| + |J\cap Y| \geq |I\cap J|.$$ This last inequality
  holds because $X$ and $Y$ are disjoint and $I\cap J \subseteq X\cup
  Y$.

  Condition (Z3) for the remaining incomparable pairs follows by
  symmetry and Proposition \ref{prop:fund_trans_delta}.  We now check
  condition (Z2).  If $I \subsetneq J$, with either $I,J \in \mcZ_A$
  or $I,J \in \mcZ_B$, then the condition holds by
  Proposition~\ref{prop:fund_trans_delta}.  If $P \subseteq I$ and $I
  \in \mcZ_A$, then
  $$r(I) - r(P) < |I\cap A| - |P\cap X| =
  |I\cap A| - |P\cap A| \le |I-P|,$$ as required.  Since $|(A\cup B) -
  P| \ge r$, the inequality in condition (Z2) holds for $P$ and $A\cup
  B$.  Checking the condition is trivial if one of the sets is
  $\varnothing$, and it follows for the remaining pairs of sets by
  symmetry.  Thus, the assertion about $\mcZ_\Delta^P$ holds.

  For $\mcZ_\Delta^Q$, let $\Delta'$ be the critical graph obtained by
  reversing the orientation of each edge of $\Delta$.  Let $(K_0',
  L_0')$ be the minimum obstruction of $\Delta'$, and let $P' = K_0'
  \cup L_0'$.  Lemma~\ref{lemma:obs_reverse} gives $P' = (A\cup B) -
  Q$.  Let $M'$ be the matroid associated to $\mcZ_{\Delta'}^{P'}$.
  We claim that the cyclic flats and their ranks for the dual $M'^*$
  are precisely those of $\mcZ_\Delta^Q$.  We use
  Proposition~\ref{prop:cfdual}.  We have
  \begin{equation*}
    r_{M'^*}(Q) = |Q| + r_{M'}(P') - r(M')
    = (2r-|P'|) + (|P'|-1)-r
    = r-1,
  \end{equation*}
  as required.  Now suppose $I' \in \mcZ(M') - \{A\cup B, P',
  \varnothing\}$ where $(A-X) \subsetneq I'$. By checking the effect of
  reversing the orientation, we get that
  \begin{equation}\label{eq:obsii}
    (A\cup B) - I' = D_\Delta(A-I'),
  \end{equation}
  which is assigned rank $|B-I'|$ in $\mcZ_\Delta^Q$.  By symmetry, it
  follows that $\mcZ(M'^*) = \mcZ_\Delta^Q$.  Also,
  \begin{equation*}
    \begin{split}
      r_{M'^*}((A\cup B) - I') &= |(A\cup B) - I'| + r_{M'}(I') - r(M')\\
      &= 2r-|I'| + |I'\cap A| -r\\
      &= r-|I'\cap B|\\
      &= |B-I'|.
    \end{split}
  \end{equation*}

  Before treating $\mcZ_\Delta^{P,Q}$, we make a general observation.
  Let $M$ be a matroid with rank function $\rho$, and fix a set $Z'
  \in 2^{E(M)} - \mcZ(M)$.  Let $\mcZ'$ be the collection $\mcZ(M)\cup
  \{Z'\}$ and suppose a function $r' \from \mcZ' \to \mathbb{N}$
  agrees with $\rho$ on $\mcZ(M)$.  Now assume that $r'$ and $\mcZ'$
  satisfy conditions (Z0) and (Z2) of Theorem~\ref{thm:cyclic_flats}.
  We claim that the inequality in condition (Z3) for $r'$ and $\mcZ'$
  follows for all sets $I, J \in \mcZ(M)$.  To see this, note that the
  presence of $Z'$ makes the join of $I$ and $J$ smaller precisely
  when $I\cup J\subseteq Z' \subsetneq I\join_{\mcZ(M)} J$, and this
  preserves the validity of condition (Z3). Also, the meet of $I$ and
  $J$ is greater precisely when $I\meet_{\mcZ(M)} J\subsetneq Z'
  \subseteq I\cap J$, and condition (Z3) follows in this case since
  condition (Z2) gives $|Z'-(I\meet_{\mcZ(M)}
  J)|>r'(Z')-r'(I\meet_{\mcZ(M)} J)$.  In all other cases, neither the
  meet nor the join changes.

  Now $(A\cup B) - (X\cup Y) \subseteq Q-P$, so $|Q-P| \ge
  r-1=r(Q)>r(Q)-r(P)$, so condition (Z2) holds for $P$ and $Q$.  By
  what we noted above, condition (Z3) for $\mcZ_\Delta^{P,Q}$ will
  follow from our work on $\mcZ_\Delta^P$ and $\mcZ_\Delta^Q$ once we
  prove condition (Z0), which we do next.

  Since $\mcZ_\Delta^{P,Q}$ is finite and has a greatest member,
  $A\cup B$, it suffices to show that meets exist.  Let $\meet_P$ and
  $\join_P$ denote the operations of the lattice $\mcZ_\Delta^P$, and
  similarly for the others.  Let $I\in \mcZ_A - \{A\cup B,
  \varnothing\}$.  Let $W=I\meet_Q Q$, which, being contained in $I$,
  is in $\mcZ_A$.  The only possible candidates for $I\meet_{P,Q} Q$
  are $W$ and $P$, and the latter is a candidate only if $P\subseteq
  I\cap Q$, so assume this inclusion holds.  Now $W\subseteq P$ if and
  only if $W=\varnothing$, in which case $I\meet_{P,Q} Q=P$, so assume
  $W\ne\varnothing$.  Thus, $P\vee_P W\in \mcZ_A$.  From $W\subseteq
  I$ and $P\subseteq I$ we get $P\vee_P W\subseteq I$.  From $W\in
  \mcZ_A$ and $W\subseteq Q$ we get $W\cap B\subseteq L_1$.  By
  equation (\ref{eq:joinIP}), $$P\vee_P W = D_\Delta(L_0\cup (W\cap
  B))\subseteq D_\Delta(L_1)\subseteq Q.$$ Thus, $I\meet_Q
  Q=W\subseteq P\vee_P W\subseteq I\cap Q$, which, since $P\vee_P W\in
  \mcZ_A$, gives $P\vee_P W=W$.  Thus, $P\subseteq W$, so
  $I\meet_{P,Q} Q=W$.  By symmetry, if $J \in \mcZ_B$, then
  $J\meet_{P,Q} Q$ exists.  It is easy to see that all other meets
  exist since $\mcZ_\Delta^P$ and $\mcZ_\Delta^Q$ are lattices, so
  conditions (Z0)--(Z3) hold for $\mcZ_\Delta^{P,Q}$.
\end{proof}

Both $A$ and $B$ are bases of the matroids whose lattices of cyclic
flats are $\mcZ_\Delta^P$, $\mcZ_\Delta^Q$, and $\mcZ_\Delta^{P, Q}$,
as we see from the inequalities $|P|-1+|A-P|>|A|$ and $r-1+|A-Q|>|A|$,
their counterparts for $B$, equation (\ref{eq:cfrank}), and
Proposition \ref{prop:fund_trans_delta}.  Also, Proposition
\ref{prop:obsge2} ensures that, for all three matroids, the
basis-exchange digraph of $A$ and $B$ is $\Delta$.  Neither
$\mcZ_\Delta^P$ nor $\mcZ_\Delta^Q$ is a suitable choice for
$M(\Delta_7)$, where $\Delta_7$ is the digraph in
Figure~\ref{fig:Delta7}, because both of the resulting matroids have
non-base-orderable proper minors.  We define $M(\Delta)$ to be the
matroid with $\mcZ(M(\Delta)) = \mcZ_\Delta^{P, Q}$, and we know of no
such $M(\Delta)$ that is not an excluded minors for base-orderability.
Another reason for choosing this definition of $M(\Delta)$ is to make
following proposition true.

\begin{proposition}
  Let $\Delta$ be a critical graph. If $\Delta'$ is the digraph
  obtained by reversing the orientation of every edge of $\Delta$,
  then $M(\Delta)^{*} = M(\Delta')$.
\end{proposition}

While Ingleton did not state his construction of $M(\Delta)$ when
$\Delta$ has an obstruction, he did state this duality result.  We
think it likely, but cannot be certain, that the matroid $M(\Delta)$
defined above is the one he intended.  With that caution, we state the
next conjecture.

\begin{conjecture}\label{conj:obs}
  If $\Delta$ has an obstruction, then $M(\Delta)$, the matroid with
  lattice of cyclic flats equal to $\mcZ_\Delta^ {P,Q}$, is an
  excluded minor for $\mathcal{BO}$.
\end{conjecture}

\subsection{Evidence for the conjectures}\label{subsec:data}

Using a computer, we have verified Conjectures~\ref{conj:noobs} and
\ref{conj:obs} for all critical graphs with $r \le 9$.  We first note
that it is straightforward to write a program to test if a matroid is
$k$-base-orderable: just check all possible bijections between all
pairs of bases.  To test the conjectures, we first generated, up to
isomorphism, all orientations of each $K_{s,t}$, where $s+t=r+1$ and
$s,t\geq 2$, using the \texttt{directg} command distributed with
Brendan McKay's \texttt{nauty} program \cite{nug}.  Next we rejected
orientations that had sources or sinks, checked if an obstruction was
present, and then constructed $M(\Delta)$ with the help of the matroid
package for SageMath \cite{sage_matroids}.  Finally, we checked
whether the single-element deletions and contractions of $M(\Delta)$
were base-orderable using a program we wrote in the C programming
language.  We discovered that, for $r \le 9$, if $\Delta$ has no
obstruction, then $M(\Delta)$ is an excluded minor for $\mathcal{BO}$
and $\mathcal{SBO}$, while if $\Delta$ has an obstruction, then
$M(\Delta)$ is an excluded minor for $\mathcal{BO}$ but not
$\mathcal{SBO}$.  Table \ref{table:computation} gives the number of
matroids $M(\Delta)$ checked this way.

\begin{table}[h]
  \begin{center}
  \begin{tabular}{|c|c|c|c|c|}
    \hline
     &   & orientations  & orientations
      &  \rule{0pt}{10pt}  \\ 
    $r$ & $K_{s,t}$  &  with no obstructions & 
    with obstructions  &  total \\ \hline
    $3$  & $K_{2,2}$  & $1$  & $0$ & $1$  \rule{0pt}{10pt}  
    \\ \hline
    $4$  & $K_{2,3}$  & $1$  & $0$ & $1$  \rule{0pt}{10pt}  
    \\ \hline
    $5$  & $K_{2,4}$  & $2$  & $0$ & $2$  \rule{0pt}{10pt}  
    \\ 
    $5$  & $K_{3,3}$  & $3$  & $0$ & $3$  \rule{0pt}{10pt}  
    \\ \hline
    $6$  & $K_{2,5}$  & $2$  & $0$ & $2$  \rule{0pt}{10pt}  
    \\ 
    $6$  & $K_{3,4}$  & $15$  & $0$ & $15$  \rule{0pt}{10pt}  
    \\ \hline
    $7$  & $K_{2,6}$  & $3$  & $0$ & $3$  \rule{0pt}{10pt}  
    \\ 
    $7$  & $K_{3,5}$  & $34$  & $0$ & $34$  \rule{0pt}{10pt}  
    \\ 
    $7$  & $K_{4,4}$  & $43$  & $1$ & $44$  \rule{0pt}{10pt}  
    \\ \hline
    $8$  & $K_{2,7}$  & $3$  & $0$ & $3$  \rule{0pt}{10pt}  
    \\ 
    $8$  & $K_{3,6}$  & $68$  & $0$ & $68$  \rule{0pt}{10pt}  
    \\ 
    $8$  & $K_{4,5}$  & $331$  & $3$ & $334$  \rule{0pt}{10pt}  
    \\ \hline
    $9$  & $K_{2,8}$  & $4$  & $0$ & $4$  \rule{0pt}{10pt}  
    \\ 
    $9$  & $K_{3,7}$  & $120$  & $0$ & $120$  \rule{0pt}{10pt}  
    \\ 
    $9$  & $K_{4,6}$  & $1111$  & $8$ & $1119$  \rule{0pt}{10pt}  
    \\ 
    $9$  & $K_{5,5}$  & $1203$  & $10$ & $1213$  \rule{0pt}{10pt}  
    \\ \hline
  \end{tabular}
  \end{center}
  \caption{The number of matroids $M(\Delta)$ checked by computer.}
  \label{table:computation}
\end{table}

\section{Complete classes of matroids}\label{sec:complete}

Recall Definition \ref{def:complete}: a class of matroids is complete
if it is closed under the operations of minors, duals, direct sums,
truncations, and induction by directed graphs.  In Section
\ref{subsec:reform}, we justify the equivalent formulation of
completeness given in Theorem \ref{thm:recastcomplete}, which better
suits our work in Section \ref{sec:kbo}.  In Section
\ref{subsec:otherops}, we discuss some properties of complete classes,
including additional operations under which they are closed.

\subsection{A reformulation of complete classes}\label{subsec:reform}
We prove the following theorem.

\begin{theorem}\label{thm:recastcomplete}
  A class of matroids is complete if and only if it is closed under
  the operations of minors, duals, direct sums, and principal
  extension.
\end{theorem}

As we justify this theorem, largely by collecting known results, we
discuss principal extension, as well as induction by both directed and
bipartite graphs.  Additional information on these topics can be found
in \cite[Sections 7.1 and 11.2]{oxleybook}.

We first review the two notions of inducing matroids.
First let $\Gamma$ be a directed graph.  Let $M$ be a matroid with $E(M)
\subseteq V(\Gamma)$. In the \emph{induced matroid} $\Gamma(M)$ on
$V(\Gamma)$, a subset of $V(\Gamma)$ is independent if and only if it
can be linked to an independent set of $M$.

Now let $M$ be a matroid, let $T$ be a set disjoint from
$E(M)$, and let $\Delta$ be a bipartite graph with bipartition $\{T,
E(M)\}$. In the \emph{induced matroid} $\Delta(M)$ on $T$, a subset of
$T$ is independent if and only if it can be matched in $\Delta$ to an
independent set of $M$.  (We caution the reader to not confuse
$\Delta(M)$ with the matroid $M(\Delta)$ defined in
Section~\ref{sec:ingletons_conjecture}.)  Thus, transversal matroids
are those that can be induced from free matroids by bipartite graphs.
For $X \subseteq T$, its rank in the induced matroid $\Delta(M)$ is
\begin{equation}\label{eq:induced}
  r_{\Delta(M)}(X) = \min \{ r_M(N(Y)) + |X - Y|\,:\,Y\subseteq X \}.
\end{equation}

Ingleton and Piff \cite{ingleton_piff} showed that if $M$ is induced
from $N$ by a directed graph, then $M^*$ is induced from $N^*$ by a
bipartite graph (see their proof of their Theorem 3.7).  This gives
the next result.

\begin{theorem}\label{thm:ingpiff}
  If a class of matroids is closed under induction by bipartite graphs
  and under duality, then it is also closed under induction by
  directed graphs.
\end{theorem}

Thus, a dual-closed, minor-closed class of matroids is closed under
induction by directed graphs if and only if it is closed under
induction by bipartite graphs.

Intuitively, we get a principal extension of a matroid by adding a
point freely to a flat.  To be precise, let $M$ be a matroid with rank
function $r$, let $Y \subseteq E(M)$, and let $e$ be an element not in
$E(M)$.  The \emph{principal extension of $M$ into $Y$}, denoted $M
+_Y e$, is the matroid on the set $E(M)\cup \sg{e}$ whose rank
function is given as follows: for $X \subseteq E(M)$, we have $r_{M+_Y
  e}(X) = r(X)$, and
\begin{equation*}
  r_{M+_Y e}(X \cup e) =
  \begin{cases}
    r(X) & \text{if} \ r(X\cup Y) = r(X),\\
    r(X) + 1 & \text{otherwise,}
  \end{cases}
\end{equation*}
or, more compactly,
\begin{equation}\label{eq:principal_ext}
  r_{M +_{Y} e}(X \cup e) = \min \{r(X)+1,\ r(X\cup Y)\}.
\end{equation}
We also say that $M +_{Y} e$ is the matroid obtained by \emph{adding
  $e$ freely to the set $Y$}.  Note that $M +_{Y} e = M+_{\cl(Y)} e$.
The \emph{free extension} of $M$ is the principal extension $M
+_{E(M)} e$.  Also, the \emph{truncation} of $M$ is $(M +_{E(M)}
e)/e$.  Thus, in order to prove Theorem \ref{thm:recastcomplete}, it
suffices to prove the next result.

\begin{theorem}\label{thm:reductionrecastcomplete}
  A class of matroids is closed under deletion and principal extension
  if and only if it is closed under induction by bipartite graphs.
\end{theorem}

Let $\phi:E(M)\to E(M')$ be an isomorphism of $M$ onto $M'$, where
$E(M)$ and $E(M')$ are disjoint.  Fix a subset $Y$ of $E(M)$ and
element $e\not\in E(M)\cup E(M')$, and define a bipartite graph
$\Delta$ with bipartition $\{E(M)\cup e,E(M')\}$ and edge set
$$\{x\,\phi(x)\,:\,x\in E(M)\}\,\cup\,\{e\,\phi(y)\,:\,y\in Y\}.$$
From equation (\ref{eq:induced}) and those above giving the rank
function of $M+_Y e$, it is routine to show that the matroid that $M'$
induces on $E(M)\cup e$ via $\Delta$ is the principal extension $M+_Y
e$.  It is easy to realize deletion via induction by a bipartite
graph, so one direction of Theorem \ref{thm:reductionrecastcomplete}
follows.

The justification of the other direction, given in Lemma
\ref{lem:masoniteratepe}, uses the next lemma, which gives the rank
function of a sequence of principal extensions of $M$ into subsets of
$E(M)$.  This lemma implies that the result of a sequence of principal
extensions of $M$ into subsets of $E(M)$ does not depend on the order.
Thus we may say that these principal extensions are performed
\emph{simultaneously}.  To make this precise, let $M$ be a matroid
with rank function $r$, and let $e_1, e_2, \ldots, e_n$ be distinct
elements not in $E(M)$.  For $1 \le i \le n$, choose a set $F_i
\subseteq E(M)$.  Define $M_0 = M$, and for $0 \le i < n$, define
$M_{i+1} = M_{i} +_{F_{i+1}} e_{i+1}$.  For simplicity, we define $r_i
= r_{M_i}$ for $1 \le i \le n$.  Thus, $M_n$ is the matroid obtained
by consecutively adding $e_i$ freely to the set $F_i$ for $1 \le i \le
n$.  In Lemma \ref{lem:masoniterateprep}, we use $[n]$ for the set
$\{1,2,\ldots,n\}$.  For $I \subseteq [n]$, we define
$$e_I = \{e_i : i \in I\} \qquad \text{and} \qquad F_I = \bigcup_{i\in
  I} F_i.$$

\begin{lemma}\label{lem:masoniterateprep}
  Using the notation above, if $i\in [n]$, $X \subseteq E(M)$, and $J
  \subseteq [i]$, then
  \begin{equation}\label{eq:comp1}
    r_i(X \cup e_J) = \min_{I \subseteq J} \{r(X \cup F_I)
    + |J - I|\}.
  \end{equation}
\end{lemma}

\begin{proof}
  Equation (\ref{eq:principal_ext}) gives the case $i=1$.  Assume that
  equation~(\ref{eq:comp1}) holds for some $i\in [n-1]$. To deduce case
  $i+1$, let $X \subseteq E(M)$ and $J \subseteq [i+1]$.  If $i+1
  \not\in J$, then we have $r_{i+1}(X \cup e_J) = r_i(X \cup e_J)$,
  from which the needed equality for $r_{i+1}(X \cup e_J)$ follows.
  Now assume $i+1 \in J$.  Set $J' = J - \{i+1\}$, so $e_J = e_{J'}
  \cup \{e_{i+1}\}$.  By equation~(\ref{eq:principal_ext}),
  \begin{equation}\label{eq:inductionA}
    \begin{split}
      r_{i+1}(X\cup e_J) &= r_{i+1}(X\cup e_{J'}\cup e_{i+1})\\
      &=\min\{r_i(X\cup e_{J'})+1,\, r_i(X\cup e_{J'}\cup F_{i+1})\}.
    \end{split}
  \end{equation}
  By the induction hypothesis, the first term, $r_i(X\cup e_{J'})+1$,
  is
  \begin{equation*}
    \min_{I'\subseteq J'}\{r(X\cup F_{I'})+|J'-I'|+1\} 
    =    \min_{I\subseteq J\,:\, i+1\not\in I}\{r(X\cup F_I)+|J-I|\} ,
  \end{equation*}
  and the second, $r_i(X\cup e_{J'}\cup F_{i+1})$, is
  \begin{equation*}
    \min_{I'\subseteq J'}\{r(X\cup F_{I'}\cup F_{i+1})+|J'-I'|\}
    = \min_{I\subseteq J\,:\, i+1\in I}\{r(X\cup F_I)+|J-I|\}.  
  \end{equation*}
  With these equalities, equation (\ref{eq:inductionA}) gives
  equation~(\ref{eq:comp1}) for $r_{i+1}(X\cup e_J)$.
\end{proof}

\begin{lemma}\label{lem:masoniteratepe}
  Let $M$ be a matroid, let $T$ be a set disjoint from $E(M)$, and let
  $\Delta$ be a bipartite graph with bipartition $\{T, E(M)\}$.  The
  induced matroid $\Delta(M)$ is obtained from $M$ by first adding
  each $t\in T$ freely to the set $N_\Delta(t)$ and then deleting
  $E(M)$.
\end{lemma}

\begin{proof}
  Write $T = \{e_1,e_2,\ldots,e_n\}$, set $F_i = N_\Delta(e_i)$, and
  define $M_n$ as above.  Comparing equation~(\ref{eq:comp1}) with $X
  = \varnothing$ to equation (\ref{eq:induced}) gives $M_n | T =
  \Delta(M)$.
\end{proof}

This completes the proofs of Theorems
\ref{thm:reductionrecastcomplete} and \ref{thm:recastcomplete}.  When
$M$ is free, Lemma \ref{lem:masoniteratepe} gives the geometric
description of transversal matroids, as in \cite[Proposition
11.2.26]{oxleybook}.  A simple variation on these ideas justifies the
remark by Mason \cite{mason_geom} that simultaneous principal
extensions may be realized by induction from a bipartite graph.

\subsection{Further properties of complete classes}\label{subsec:otherops}

The class of gammoids is complete.  As Ingleton
\cite{ingleton_related} observed, since transversal matroids are those
induced from free matroids via bipartite graphs, the class of gammoids
is the smallest complete class.

\begin{proposition}\label{prop:allgammoids}
  Every non-empty complete class of matroids contains all gammoids.
\end{proposition}

Other examples of complete classes of matroids include:
$\mathcal{BO}$, $\mathcal{SBO}$, $k$-$\mathcal{BO}$ (treated in
Section~\ref{sec:kbo}); matroids representable over fields of a given
characteristic (see \cite{piff_welsh}); and the class of matroids with
no $M(K_4)$-minor (see \cite{sims_complete}).

Complete classes are closed under all operations that arise by
combining those under which they are already known to be closed.  For
example, since the matroid union, $M \vee N$, is obtained by induction
from the direct sum $M \oplus N$ by a certain bipartite graph (see
\cite[Theorem 11.3.1]{oxleybook}), complete classes are closed under
matroid union.  The same holds for the free product and all principal
sums (which are special matroid unions; see
\cite{bonin_kung_semidirect}).  Closure under parallel connections
follows from the description Mason \cite{mason_geom} gave of this
operation, which we recall.  Let $M$ and $N$ be matroids with $E(M)
\cap E(N) = \{p\}$.  To obtain their parallel connection, $P(M,N)$, in
$N$ replace $p$ by $p_N$, giving $N'$, where $p_N\not \in E(M) \cup
E(N)$; then $P(M, N)$ is $$\big((M \oplus N') +_{\{p, p_N\}} e\big)
\contract e \delete p_N.$$ This applies even if $p$ is a loop or
coloop of either $M$ or $N$.  It now follows that complete classes are
also closed under series connection (the dual of parallel connection)
and $2$-sums.

While they play no role in this paper, we conclude this section with
two observations.

Recall that a $2$-connected matroid is not $3$-connected if and only
if it is a $2$-sum of two of its proper minors.  Thus, the results
above imply that any excluded minor for a complete class of matroids
must be $3$-connected.

Second, we note that in Theorem~\ref{thm:recastcomplete}, we may
replace closure under principal extensions by closure under principal
extensions into sets of at most two elements.  This result, which is
used in \cite{piff_welsh} and \cite{sims_complete}, follows by
repeatedly applying Lemma~\ref{lem:2willdo}, whose straightforward
proof we omit.

\begin{lemma}\label{lem:2willdo}
  Let $M$ be a matroid with rank function $r$.  Let $F \subseteq E(M)$
  and $G \subseteq F$, and let $e_1$ and $e_2$ be points not in
  $E(M)$.  Then
  $$M +_F e_1 = ((M +_G e_2) +_{(F-G) \cup e_2} e_1)\delete e_2.$$
\end{lemma}

\section{$k$-$\mathcal{BO}$ is a complete class}\label{sec:kbo}

The main result of this section is the following theorem.

\begin{theorem}\label{thm:kbocomplete}
  For a fixed $k \ge 1$, the class of $k$-base-orderable matroids is
  complete.
\end{theorem}

This theorem implies the known results that $\mathcal{BO}$ and
$\mathcal{SBO}$ are complete, and combining it with Proposition
\ref{prop:allgammoids} gives another proof that gammoids are strongly
base-orderable.  (For a short direct proof, see \cite[Theorem
42.11]{schrijverB}.)  Also, $k$-$\mathcal{BO}$ is closed under all the
operations discussed in Section \ref{subsec:otherops}.  In particular,
we recover the result of Brualdi \cite{brualdi_induced} that
$\mathcal{BO}$ is closed under induction by directed graphs.

By Proposition \ref{prop:kbobasics}, the class $k$-$\mathcal{BO}$ is
closed under direct sums, duals, and minors.  To complete the proof of
Theorem \ref{thm:kbocomplete}, we address closure under principal
extensions in Proposition \ref{prop:kboPrExt}.  The following
well-known result (see, e.g., \cite[Problem 7.2.4(a)]{oxleybook})
gives the bases of a principal extension.

\begin{lemma}\label{lem:basesofprext}
  Let $M$ be a matroid, with $\mathcal{B}(M)$ its set of bases, let
  $F$ be a flat of $M$, and let $e$ be an element not in $E(M)$.  The
  set of bases of $M +_{F} e$ is
  $$\mathcal{B}(M) \cup \{(B - \sg{f}) \cup
  \sg{e} : B \in \mathcal{B}(M) \text{ and } f \in B \cap F\}.$$
\end{lemma}

\begin{proposition}\label{prop:kboPrExt}
  If $M$ is $k$-base-orderable, then so is any principal extension $M
  +_{F} e$.
\end{proposition}

\begin{proof}
  Let $B_1$ and $B_2$ be bases of $M +_{F} e$.  If $e \not\in B_1\cup
  B_2$, then there is a $k$-exchange-ordering $\sigma \from B_1 \to
  B_2$ by assumption.

  If $e \in B_1- B_2$, then there is a basis $(B_1 - \sg{e}) \cup
  \sg{x}$ of $M$ with $x \in F$, and a $k$-exchange-ordering $\sigma
  \from (B_1 - \sg{e}) \cup \sg{x} \to B_2$.  Define $\tau \from B_1
  \to B_2$ by
  \begin{equation*}
    \tau(z) =
    \begin{cases}
      \sigma(z) & \text{if} \ z \ne e,\\
      \sigma(x) & \text{if} \ z = e.\\
    \end{cases}
  \end{equation*}
  Since $e$ can replace $x$ in any basis of $M$ to yield a basis of $M
  +_{F} e$, it follows that $\tau$ is a $k$-exchange-ordering for $M
  +_{F} e$.  The case with $e \in B_2- B_1$ follows by symmetry.

  Now suppose $e \in B_1\cap B_2$.  There are bases $B_{1x}=(B_1 -
  \sg{e}) \cup \sg{x}$ and $B_{2y}=(B_2 - \sg{e}) \cup \sg{y}$ of $M$
  with $x, y \in F$, and a $k$-exchange-ordering $\sigma \from B_{1x}
  \to B_{2y}$.  If $\sigma(x) = y$, then the bijection $\tau \from B_1
  \to B_2$ given by
  \begin{equation*}
    \tau(z) =
    \begin{cases}
      \sigma(z) & \text{if} \ z \ne e,\\
      e & \text{if} \ z = e,\\
    \end{cases}
  \end{equation*}
  is a $k$-exchange-ordering for $M +_{F} e$.

  The rest of the proof treats the case with $\sigma(x) \ne y$ and
  uses the following notation: $B_1 = \{e, a_2, a_3, \dots, a_r\}$ and
  $B_2 = \{e, b_1, b_3, \dots, b_r\}$, and $\sigma$ is given by
  $\sigma(x) = b_1$, $\sigma(a_2) = y$, and $\sigma(a_j) = b_j$ for $j
  \ge 3$; we abbreviate this by
  $$\sigma = \left(
    \begin{array}{@{\;}c@{\ }c@{\ }c@{\ }c@{\ }c@{\;}}
      x&a_2&a_3&\dots&a_r\\
      b_1&y&b_3&\dots&b_r
    \end{array}\right).$$
  Define $\tau \from B_1 \to B_2$ to fix $e$, map $a_2$ to $b_1$, and
  agree with $\sigma$ on $a_3,\ldots,a_r$, that is,
  $$\tau = \left(
    \begin{array}{@{\;}c@{\ }c@{\ }c@{\ }c@{\ }c@{\;}}
      e&a_2&a_3&\dots&a_r\\
      e&b_1&b_3&\dots&b_r
    \end{array}\right).$$
  We claim that $\tau$ is a $k$-exchange-ordering for $M +_{F} e$.
  Let $X \subseteq B_1$ with $|X| \le k$.  The case with $X\cap
  \{e,a_2\}= \varnothing$ is immediate. If both $e\in X$ and $a_2 \in
  X$, then exchanging $(X-e)\cup x$ and $\sigma((X-e)\cup x)$ in
  $B_{1x}$ and $B_{2y}$, and then using Lemma \ref{lem:basesofprext},
  shows that both $(B_1 - X) \cup \tau(X)$ and $(B_2 - \tau(X)) \cup
  X$ are bases of $M +_{F} e$.  We get the same conclusion if $e\in X$
  and $a_2 \not\in X$ by exchanging $X-x$ and $\sigma(X-x)$ in
  $B_{1x}$ and $B_{2y}$, and then using Lemma \ref{lem:basesofprext}
  (note the inequality $|X|\leq k$ in Definition \ref{def:kbo}).

  Finally, assume that $e \not\in X$ and $a_2 \in X$.  By relabeling,
  we may assume that $X$ is $\{a_2, a_3, \dots, a_i\}$, so $i \le
  k+1$.  First assume $x=b_1$.  Since $\sigma$ is a
  $k$-exchange-ordering, $\{x, y, b_3, \dots, b_i, a_{i+1}, \dots,
  a_r\}$ must be a basis of $M$, and we may replace $y$ with $e$ to
  get that $\{e, x, b_3, \dots, b_i, a_{i+1}, \dots, a_r\}$ is a basis
  of $M +_{F} e$.  To show that $\{e, a_2, \dots, a_i, b_{i+1}, \dots,
  b_r\}$ is also a basis of $M +_{F} e$, note that
  $$(B_{2y} - \sigma(X)) \cup X =
  \{b_1, a_2, \dots, a_i, b_{i+1}, \dots, b_r\} = \{x, a_2, \dots,
  a_i, b_{i+1}, \dots, b_r\}$$ is a basis of $M$, and then replace $x$
  with $e$.  The case where $y = a_2$ is similar, though not
  symmetric, to that of $x = b_1$.  To provide the details, assume $y
  = a_2$.  Now let $X' = (X-\sg{a_2}) \cup \sg{x}$.  Since $\sigma$ is
  a $k$-exchange-ordering, we have that
  $$(B_{1x} - X')\cup \sigma(X') = \{b_1, y, b_3, \dots, b_i, a_{i+1},
  \dots, a_r\}$$ is a basis of $M$, and replacing $y$ with $e$ gives
  that $\{e, b_1, b_3, \dots, b_i, a_{i+1}, \dots, a_r\}$ is a basis
  of $M +_F e$.  We also have that the following set is a basis of $M
  +_F e$:
  $$(B_{2y} - \sigma(X'))\cup X' = \{x, y, a_3, \dots, a_i, b_{i+1},
  \dots, b_r\}.$$ Thus replacing $x$ by $e$ and recalling that $y=a_2$
  gives that $\{e, a_2, \dots, a_i, b_{i+1}, \dots, b_r\}$ is a basis
  of $M$.

  So we assume that $x \ne b_1$ and $y \ne a_2$ for the rest of the
  proof.  We next show that $$(B_1 - X) \cup \tau(X) = \{e, b_1, b_3,
  \dots, b_i, a_{i+1}, \dots, a_r\}$$ is a basis of $M +_{F} e$.
  Assume the contrary.  Then, since $x \ne b_1$ 
 and
  $\sigma$ is a $k$-exchange-ordering, it follows that both $$\{x,
  b_1, b_3, \dots, b_i, a_{i+1}, \dots, a_r\} \text{ and } \{y, b_1,
  b_3, \dots, b_i, a_{i+1}, \dots, a_r\}$$ have size $r$, and, hence,
  they are dependent in $M$.  Therefore they contain circuits, say
  $C_x$ and $C_y$, respectively. Since $\sigma$ is a
  $k$-exchange-ordering, it follows that both $$B' = \{x, y, b_3,
  \dots, b_i, a_{i+1}, \dots, a_r\} \text{ and } \{b_1, a_2, b_3,
  \dots, b_i, a_{i+1}, \dots, a_r\}$$ are bases of $M$.  Therefore, it
  must be that $\{x, b_1\} \subseteq C_x$ and $\{y, b_1\} \subseteq
  C_y$.  By circuit elimination, there is a circuit $C$ of $M$ such
  that $C \subseteq (C_x \cup C_y) - b_1$.  Since this last set is a
  subset of $B'$, we have reached a contradiction.  Thus, $(B_1 - X)
  \cup \tau(X)$ is a basis of $M +_{F} e$. A similar argument shows
  that $(B_2 - \tau(X)) \cup X$ is also a basis of $M +_{F} e$.
\end{proof}

We next point out that there are other complete classes along these
lines.

\begin{definition}
  Let $M$ be a matroid, and let $0 \le k \le r(M)$ and $0 \le l \le
  r(M)$ with $k+l>0$.  We say $M$ is \emph{$(k,l)$-base-orderable} if,
  given any two bases $B_1$ and $B_2$, there is a bijection $\sigma
  \from B_1 \to B_2$ so that for every $X \subseteq B_1$ with $|X| \le
  k$ or $|X| \ge r(M) - l$, the set $(B_1 - X) \cup \sigma(X)$ is a
  basis.
\end{definition}

It follows from the definition that for $k\ge 1$, a matroid is
$k$-base-orderable if and only if it is $(k,k)$-base-orderable.  It
also follows that a matroid $M$ is $(k,l)$-base-orderable if and only
if it is $(l,k)$-base-orderable.  Note that every matroid is
$(1,0)$-base-orderable by the bijective-exchange property.  However,
matroids do not in general satisfy a multiple bijective-exchange property.
For example, since $r(M(K_4)) = 3$ and $M(K_4)$ is not base-orderable, it
follows that it is not $(2,0)$-base-orderable.  It is easy to modify
the proofs in this section to show the following strengthening.

\begin{theorem}
  For fixed $k$ and $l$, the class of $(k,l)$-base-orderable matroids
  is complete.
\end{theorem}

In a different direction, we close this section by showing that
$\mathcal{BO}$ is closed under circuit-hyperplane relaxation.  We do
not yet know whether the same holds for $\mathcal{SBO}$,
$k$-$\mathcal{BO}$, or all complete classes of matroids.

\begin{proposition}
  Let $M'$ be the relaxation of a matroid $M$ by a circuit-hyperplane
  $X$ of $M$.  If $M$ is base-orderable, then so is $M'$.
\end{proposition}

\begin{proof}
  Let $B_1$ and $B_2$ be bases of $M'$.  It suffices to consider the
  case where $B_2 = X$.  By the bijective exchange property, there is
  a bijection $\sigma \from B_1 \to X$ such that for all $y \in B_1$,
  the set $(B_1 - \sg{y}) \cup \sg{\sigma(y)}$ is also a
  basis. Clearly $\sigma$ fixes any element of $B_1 \cap X$.  Let $y
  \in B_1$.  To show that $\sigma$ is an exchange-ordering, we must
  show that $(X - \sg{\sigma(y)}) \cup \sg{y}$ is a basis of $M'$.
  This holds because $X-\sg{\sigma(y)}$ is an independent hyperplane
  of $M'$.
\end{proof}

\section{An infinite family of excluded minors for gammoids and for 
  $\mathcal{BO}$}\label{sec:em_gammoids}

Ingleton~\cite{ingleton_nonbo} stated (without giving his proof) that
if $\Delta$ is a critical graph and (in the notation of Definition
\ref{def:ingleton_critical}) either $|X|$ or $|Y|$ is two, then
$M(\Delta)$ is an excluded minor for $\mathcal{BO}$.  Since $\Delta$
has neither a source nor a sink, Ingleton's hypothesis implies that
each of $X$ and $Y$ can be partitioned into two sets so that all edges
between a given block of $X$ and one of $Y$ are oriented the same way.
In this section we extend Ingleton's result to the case where there
are such partitions of $X$ and $Y$, even if $\min\{|X|, |Y|\}>2$, as
in Figure \ref{fig:gammoid_eminor}.  We show more: all single-element
contractions of such matroids $M(\Delta)$ are transversal, and all
single-element deletions are cotransversal.  Besides verifying
infinitely many more cases of Conjecture \ref{conj:noobs}, this shows
that these matroids are also excluded minors for the class of
gammoids, and for $\mathcal{SBO}$ and $k$-$\mathcal{BO}$.  Such
critical graphs $\Delta$ look like generalizations of the graph in
Figure \ref{fig:Delta3}, so we may view the matroids $M(\Delta)$ as
generalizations of $M(K_4)$.

In this section, unlike Sections \ref{sec:bedigraph} and
\ref{sec:ingletons_conjecture}, the sets $A$ and $B$ are not bases;
rather they are two sets in a $6$-tuple of sets.  Specifically, for an
integer $r\geq 3$, let $\alpha = (A, B, C, D, E, F)$ be a $6$-tuple of
disjoint nonempty sets with $$r = |A\cup B\cup C| = |D\cup E\cup
F|\qquad \text{and} \qquad r+1=|A\cup B\cup D\cup E|.$$ Let
$\Delta_\alpha$ be the directed bipartite graph with bipartition
$\{A\cup B\cup C, D\cup E\cup F\}$ having the following edges:
\begin{enumerate}
\item $(a,d)$ for all $a\in A$ and $d\in D$,
\item $(e,a)$ for all $a\in A$ and $e\in E$,
\item $(d,b)$ for all $b\in B$ and $d\in D$, and 
\item $(b,e)$ for all $b\in B$ and $e\in E$. 
\end{enumerate}
Thus, $\Delta_\alpha$ is a critical graph with no obstructions.  From
Proposition~\ref{prop:noobs}, in the associated rank-$r$ matroid
$M(\Delta_\alpha)$, which we shorten to $M_\alpha$, the proper
nonempty cyclic flats are $C\cup B\cup E$, \, $C\cup A\cup D$, \,
$F\cup E\cup A$, \, and $F\cup D\cup B$, and their ranks are given by
\begin{equation}\label{eq:malphaZ}
  \begin{split}
    r(C\cup B\cup E) & = |C| + |B|, \\
    r(C\cup A\cup D) & = |C| + |A|, \\
    r(F\cup E\cup A) & = |F| + |E|,  \text{ and }  \\
    r(F\cup D\cup B) & = |F| + |D|.
  \end{split}
\end{equation}
Figure~\ref{fig:gammoid_eminor} gives an example.

\begin{figure}[h]
  \centering 
\begin{tikzpicture}[scale=1]
  \filldraw (0,0) circle (3pt);
  \filldraw (1,0) circle (3pt);
  \filldraw (2,0) circle (3pt);
  \filldraw (3,0) circle (3pt);
  \filldraw (4,0) circle (3pt);
  \filldraw (5,0) circle (3pt);
  \filldraw (6,0) circle (3pt);
  \filldraw (7,0) circle (3pt);

  \filldraw (0,1.75) circle (3pt);
  \filldraw (1,1.75) circle (3pt);
  \filldraw (2,1.75) circle (3pt);
  \filldraw (3,1.75) circle (3pt);
  \filldraw (4,1.75) circle (3pt);
  \filldraw (5,1.75) circle (3pt);
  \filldraw (6,1.75) circle (3pt);
  \filldraw (7,1.75) circle (3pt);

  \draw (6,0) ellipse (1.35cm and 0.35cm); 
  \draw (2.5,0) ellipse (1.85cm and 0.35cm);
  \draw (0,0) ellipse (0.35cm and 0.35cm);
  \draw (5.5,1.75) ellipse (1.85cm and 0.35cm); 
  \draw (2.5,1.75) ellipse (0.85cm and 0.35cm);
  \draw (0.5,1.75) ellipse (0.85cm and 0.35cm);

 \draw[middlearrow_var](0.12,0.4) -- (0.4,1.35);
 \draw[middlearrow_var](0.7,1.35) -- (2.3,0.4);
 \draw[middlearrow_var](2.5,0.4) -- (2.5,1.35);
 \draw[middlearrow_var](2.3,1.35) -- (0.3,0.35);

  \node at (0,-0.7) {$A$};
  \node at (2.5,-0.7) {$B$};
  \node at (6,-0.7) {$C$};
  \node at (0.5,2.45) {$D$};
  \node at (2.5,2.45) {$E$};
  \node at (5.5,2.45) {$F$};
   
\end{tikzpicture}
  \caption[A basis-exchange digraph of an excluded minor for
  gammoids.]
  {An example, with $r=8$, of the digraph $\Delta_\alpha$ described
    above.  An arrow from block $U$ to block $V$ means that there is a
    directed edge $(u, v)$ for every $u \in U$ and $v \in V$.}
  \label{fig:gammoid_eminor}
\end{figure}
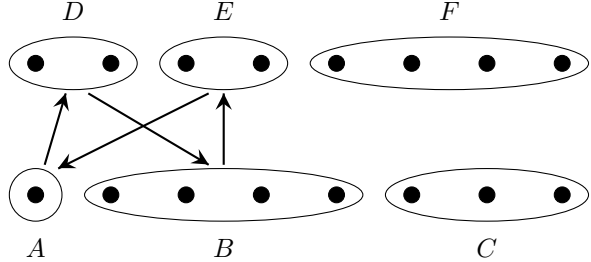

\begin{theorem}\label{thm:k4like}
  The matroid $M_\alpha$ defined above has the following properties:
  \begin{enumerate}
  \item it is not base-orderable,
  \item each of its single-element contractions is transversal,
  \item each of its single-element deletions is cotransversal, and
  \item it is an excluded minor for the following classes of matroids:
    gammoids, $\mathcal{BO}$, $\mathcal{SBO}$, and $k$-$\mathcal{BO}$
    for any $k\geq 1$.
  \end{enumerate}
\end{theorem}

\begin{proof}
  Note that it suffices to prove the first three assertions since they
  imply the last.  Also, assertion (1) follows from our work in
  Section \ref{sec:ingletons_conjecture}.

  We now prove assertion (2).  For an element $x \in E(M_\alpha)$ and
  antichain $\mcF$ of $\mcZ(M_\alpha)$ with $|\mcF| \ge 3$, let
  $\mcF_x = \{F-x : F \in \mcF\}$.  Some sets in $\mcF_x$ might not be
  cyclic flats of $M_\alpha\contract x$, but by
  Lemma~\ref{lem:cfminor}, any antichain of at least three cyclic
  flats of $M_\alpha\contract x$ is equal to some $\mcF_x$.  To show
  that $M_\alpha \contract x$ is transversal, by
  Corollary~\ref{cor:micond} it suffices to check inequality
  (\ref{eq:micond}) for all sets $\mcF_x$; to do this efficiently, we
  simplify inequality (\ref{eq:micond}) for each antichain $\mcF$ of
  $\mcZ(M_\alpha)$, and compare the result to its counterpart for
  $\mcF_x$ in $M_\alpha \contract x$.

  First note that the union of any two proper, non-empty cyclic flats
  of $M_\alpha$ has rank $r$, and that the intersection of any three
  is empty.  Thus, $\cap\mcF=\varnothing$, so
  $\cap\mcF_x=\varnothing$.

  We first let $\mcF$ consist of all four proper, nonempty cyclic
  flats of $M_\alpha$.  From the ranks given in
  equations~(\ref{eq:malphaZ}), the alternating sum in inequality
  (\ref{eq:micond}) simplifies to
  \begin{equation}\label{eq:malpha_misimp1}
    2|C|+|A|+|B|+2|F|+|D|+|E|-3r,
  \end{equation}
  which is $|C|+|F|-r = -1$.  In $M_\alpha \contract x$, the term
  $-3r$ is replaced by $-3(r-1)=-3r+3$, and, since $x$ is in at most
  two cyclic flats of $M_\alpha$, the rank of at most two sets in
  $\mcF_x$ goes down by $1$ compared to their counterparts in $M$; so
  the counterpart, in $M_\alpha \contract x$, of sum
  (\ref{eq:malpha_misimp1}) is nonnegative, as needed.  By the
  symmetry between the proper, nonempty cyclic flats of $M_\alpha$, to
  check triples, it suffices to consider the antichain
  $$\mcF=\{C\cup B\cup E, \,\, C\cup A\cup D, \,\,F\cup E\cup
  A\}.$$ The alternating sum in inequality (\ref{eq:micond}) for
  $\mcF$ in $M_\alpha$ simplifies to
  \begin{equation}\label{eq:malpha_misimp2}
    2|C|+|A|+|B|+|F|+|E|-2r,
  \end{equation}
  that is, $|C|+|F|+|E|-r$, or $|C|-|D|$.  Since $|C|+1 = |D|+|E|$,
  sum (\ref{eq:malpha_misimp2}) equals $|E|-1$, which is nonnegative.
  In $M_\alpha \contract x$, the term $-2r$ is replaced by
  $-2(r-1)=-2r+2$, and the rank of at most two sets in $\mcF_x$ goes
  down by $1$ compared to their counterparts in $M$; so the
  counterpart of sum (\ref{eq:malpha_misimp2}) in $M_\alpha \contract
  x$ is nonnegative, as needed.  Thus, statement (2) holds.

  Assertion (3) follows by applying assertion (2) to the dual, which
  is $M(\Delta'_\alpha)$ where $\Delta'_\alpha$ reverses the
  orientation of each edge of $\Delta$.
\end{proof}

In contrast, single-element deletions of an $M_\alpha$ need not be
transversal.  For example, if
$$\alpha = (\{a_1, a_2\}, \{b_1\}, \{c_1, c_2\},
\{d_1, d_2\}, \{e_1\}, \{f_1, f_2\}),$$ then $M_\alpha \delete a_1$ is
not transversal.

Note that for a given integer $r\ge 3$, this construction yields at
least as many distinct excluded minors $M_\alpha$ as integer
partitions of $r+1$ with four parts, for which $\binom{r}{3}/4!$ is a
crude lower bound.

\section{An infinite family of excluded minors for
  $\mathcal{SBO}$}\label{sec:em_sbo}

Ingleton, in \cite{ingleton_nonbo}, was the first to exhibit a matroid
that is in $\mathcal{BO}$ but not in $\mathcal{SBO}$. Here we
generalize his construction: for a fixed integer $k \ge 2$, we
construct a family of matroids that are in $(k-1)$-$\mathcal{BO}$ but
are excluded minors for $k$-$\mathcal{BO}$ and $\mathcal{SBO}$.  When
$k = 2$, we recover Ingleton's example.  Taken together, i.e., as $k$
ranges over all integers exceeding one, these matroids form an
infinite antichain of base-orderable matroids that are excluded minors
for $\mathcal{SBO}$.

Let $k \ge 2$ be an integer and let $\beta = (A, B, C, D, E, F)$ be a
$6$-tuple of disjoint nonempty sets with $k = |C| = |F| = |A\cup B| =
|D\cup E|$.  We will define a rank-$2k$ matroid $M_\beta$ on the union
of these sets, which we (prematurely) denote $E(M_\beta)$, in terms of
cyclic flats and their ranks.  Define a function $r$ on seven subsets
of $E(M_\beta)$ as follows:
\begin{equation}\label{eq:mbetaZ}
  \begin{split}
    r(E(M_\beta)) & = 2k, \\
    r(A\cup B\cup D\cup E) & = 2k - 1,\\
    r(C\cup B\cup E) & = k + |B|, \\
    r(C\cup A\cup D) & = k + |A|, \\
    r(F\cup E\cup A) & = k + |E|, \\
    r(F\cup D\cup B) & = k + |D|, \; \text{and}  \\
    r(\varnothing) & = 0.
  \end{split}
\end{equation}
Let $\mcZ$ consist of the sets on which $r$ has been defined.

\begin{proposition}
  Let $r$, $\beta$, and $\mathcal{Z}$ be as above.  The function $r$
  can be extended to all subsets of $E(M_\beta)$ to be the rank
  function of a matroid $M_\beta$ with $\mathcal{Z}(M_\beta) =
  \mathcal{Z}$.
\end{proposition}

\begin{proof}
  We check the conditions in Theorem~\ref{thm:cyclic_flats}.
  Condition (Z1) holds by construction.  Fix $c\in C$ and $f\in F$,
  and set $\alpha=(A,B,C-c,D,E,F-f)$.  Note that $\alpha$ satisfies
  the assumptions in Section \ref{sec:em_gammoids} with $r=2k-1$, so
  we can let $M_\alpha$ be the matroid defined in that section.  The
  sets and ranks in equation (\ref{eq:mbetaZ}), apart from $A\cup
  B\cup D\cup E$, are obtained from the cyclic flats of $M_\alpha$ by
  adjoining $c$ to the sets that contain $C-c$, and $f$ to the sets
  that contain $F-f$, and increasing the rank of each augmented set by
  $1$.  From this observation and the fact that $A\cup B\cup D\cup E$
  is comparable only to $\varnothing$ and $E(M_\beta)$, it follows
  that condition (Z0) holds, and that conditions (Z2) and (Z3) hold
  for all pairs that do not include $A\cup B\cup D\cup E$.  It is
  routine to check the remaining requirements, namely, that conditions
  (Z2) and (Z3) hold for the pairs that include $A\cup B\cup D\cup E$.
\end{proof}

It follows from equations (\ref{eq:cfrank}) and (\ref{eq:mbetaZ}) that
both $A \cup B \cup C$ and $D \cup E \cup F$ are bases of $M_\beta$.
Their basis-exchange digraph has the form illustrated in
Figure~\ref{fig:sbo_eminor}.

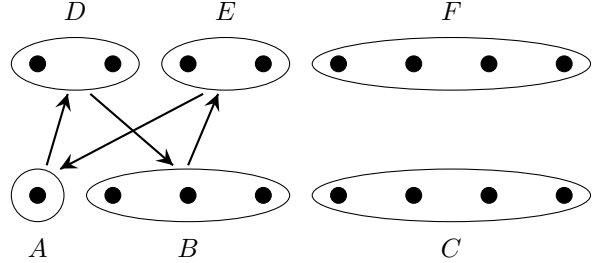
\begin{figure}[h]
\begin{tikzpicture}[scale=1]
  \filldraw (0,0) circle (3pt);
  \filldraw (1,0) circle (3pt);
  \filldraw (2,0) circle (3pt);
  \filldraw (3,0) circle (3pt);
  \filldraw (4,0) circle (3pt);
  \filldraw (5,0) circle (3pt);
  \filldraw (6,0) circle (3pt);
  \filldraw (7,0) circle (3pt);

  \filldraw (0,1.75) circle (3pt);
  \filldraw (1,1.75) circle (3pt);
  \filldraw (2,1.75) circle (3pt);
  \filldraw (3,1.75) circle (3pt);
  \filldraw (4,1.75) circle (3pt);
  \filldraw (5,1.75) circle (3pt);
  \filldraw (6,1.75) circle (3pt);
  \filldraw (7,1.75) circle (3pt);

  \draw (5.5,0) ellipse (1.85cm and 0.35cm); 
  \draw (2,0) ellipse (1.35cm and 0.35cm);
  \draw (0,0) ellipse (0.35cm and 0.35cm);

  \draw (5.5,1.75) ellipse (1.85cm and 0.35cm); 
  \draw (2.5,1.75) ellipse (0.85cm and 0.35cm);
  \draw (0.5,1.75) ellipse (0.85cm and 0.35cm);

 \draw[middlearrow_var](0.12,0.4) -- (0.4,1.35);
 \draw[middlearrow_var](0.7,1.35) -- (1.8,0.4);
 \draw[middlearrow_var](2,0.4) -- (2.4,1.35);
 \draw[middlearrow_var](2.2,1.35) -- (0.3,0.35);

  \node at (0,-0.7) {$A$};
  \node at (2,-0.7) {$B$};
  \node at (5.5,-0.7) {$C$};

  \node at (0.5,2.45) {$D$};
  \node at (2.5,2.45) {$E$};
  \node at (5.5,2.45) {$F$};
   
\end{tikzpicture}
  \caption[A basis-exchange digraph of an excluded minor for
  $\mathcal{SBO}$.]
  {An example, with $k=4$, of a basis-exchange digraph of a matroid
    $M_\beta$.  Arrows between sets are interpreted as in Figure
    \ref{fig:gammoid_eminor}.}
  \label{fig:sbo_eminor}
\end{figure}

To prove the next lemma, we use the same idea as in the proof of
Theorem~\ref{thm:k4like}.

\begin{lemma}\label{lem:sboeminor}
  Every single-element contraction of $M_\beta$ is transversal.
\end{lemma}

\begin{proof}
  For an element $x \in E(M_\beta)$ and antichain $\mcF$ of cyclic
  flats of $M_\beta$ with $|\mcF| \ge 3$, let $\mcF_x = \{F-x : F \in
  \mcF\}$.  To prove that $M_\beta \contract x$ is transversal, it
  suffices to verify inequality (\ref{eq:micond}) for all such
  $\mcF_x$ in $M_\beta \contract x$; we do this by comparing that
  inequality to its counterpart for $\mcF$ in $M_\beta$.  Symmetry
  reduces the argument to the seven cases for $\mcF$ treated below.
  In each case, we use the equality $r=2k=|A\cup B\cup D\cup E|$ as
  well as equations~(\ref{eq:mbetaZ}).  We also use the observations
  that the union of any two proper, non-empty cyclic flats of
  $M_\beta$ has rank $r=2k$, and if $|\mcF|>3$, then $\cap \mcF =
  \varnothing$, and so $\cap \mcF_x = \varnothing$.

  Let $\mcF$ consist of all five proper nonempty cyclic flats of
  $M_\beta$.  The alternating sum in inequality~(\ref{eq:micond})
  simplifies to
  \begin{equation}\label{eq:mbeta_misimp1}
    6k+|A|+|B|+|E|+|D|-1-4r,
  \end{equation}
  which equals $-1$.  In $M_\beta \contract x$, the counterpart of
  sum~(\ref{eq:mbeta_misimp1}) is nonnegative, as needed, since
  $-4(r-1)=-4r+4$ replaces $-4r$, and, with $x$ in at most three
  cyclic flats of $M_\beta$, the rank of at most three sets in
  $\mcF_x$ goes down by $1$ compared to their counterparts in $M$.

  Next consider $\mcF=\{C\cup B\cup E, \,\, C\cup A\cup D, \,\, F\cup
  E\cup A, \,\, F\cup D\cup B\}$. The alternating sum in inequality
  (\ref{eq:micond}) simplifies to
  \begin{equation*}
    4k+|A|+|B|+|E|+|D|-3r,
  \end{equation*}
  which equals $0$.  The counterpart in $M_\beta \contract x$ is also
  nonnegative since $-3(r-1)$ replaces $3r$ and, with $x$ in two
  cyclic flats of $M_\beta$, the rank of two sets in $\mcF_x$ goes
  down by $1$ compared to their counterparts in $M$.

  If $|\mcF| = 4$ and $A\cup B\cup D\cup E \in \mcF$, then by symmetry
  it suffices to consider
  $$\mcF=\{A\cup B\cup D\cup E, \,\, C\cup B\cup E, \,\, C\cup A\cup
  D, \,\, F\cup E\cup A\}.$$ The alternating sum in inequality
  (\ref{eq:micond}) for $\mcF$ in $M_\beta$ simplifies to
  $5k-1+|B|+|A|+|E|-3r$.  This equals $|E|-1$, which is nonnegative.
  This case is completed as above by noting that $x$ is in at most
  three cyclic flats of $M_\beta$

  If $|\mcF| = 3$ but $A\cup B\cup D\cup E \not\in \mcF$, then by
  symmetry it suffices to consider
  $$\mcF=\{C\cup B\cup E, \,\, C\cup A\cup D, \,\,F\cup E\cup A\}.$$
  In this case, $\cap\mcF = \varnothing$, and the alternating sum in
  inequality (\ref{eq:micond}) for $\mcF$ in $M_\beta$ simplifies to
  \begin{equation}\label{eq:mbeta_misimp4}
    3k+|A|+|B|+|E|-2r,
  \end{equation}
  which equals $|E|$. In $M_\beta \contract x$, the term $-2(r-1)$
  replaces $-2r$, and the rank of at most two sets in $\mcF_x$ goes
  down by $1$ compared to their counterparts in $M$; so the
  counterpart of sum~(\ref{eq:mbeta_misimp4}) in $M_\beta \contract x$
  is positive.

  When $|\mcF| = 3$ and $A\cup B\cup D\cup E \in \mcF$, then by
  symmetry, it suffices to assume that $C\cup A\cup D \in \mcF$ and to
  examine the three remaining cases.

  First let $\mcF=\{A\cup B\cup D\cup E, \,\, C\cup A\cup D, \,\,
  C\cup B\cup E\}$. In this case, $\cap\mcF = \varnothing$, and the
  alternating sum in inequality (\ref{eq:micond}) for $\mcF$ in
  $M_\beta$ simplifies to $2r-1+|A|+|B|-2r$, which equals $k-1$.  This
  case follows as above by noting that $x$ is in at most two cyclic
  flats of $\mcF_x$.

  Next, let $\mcF=\{A\cup B\cup D\cup E, \,\, C\cup A\cup D, \,\,
  F\cup E\cup A\}$. In this case, $\cap\mcF = A$, which is
  independent, and the alternating sum in inequality (\ref{eq:micond})
  for $\mcF$ in $M_\beta$ simplifies to $$4k+|A|+|E|-1-2r,$$ that is,
  $|A|+|E|-1$.  In $M_\beta \contract x$, the term $-2(r-1)$ replaces
  $-2r$.  If $x \not\in A$, then $x$ is in at most two cyclic flats of
  $\mcF$, so the rank of at most two sets in $\mcF_x$ goes down by
  $1$; also, $r_{M_\beta \contract x}(\cap\mcF_x) = r_{M_\beta
    \contract x}(A) \le |A|$.  If $x \in A$, then the rank of all
  three sets in $\mcF_x$ goes down by $1$; also, $r_{M_\beta \contract
    x}(\cap\mcF_x) = r_{M_\beta \contract x}(A-\sg{x}) = |A|-1$.
  Either way, the counterpart of inequality (\ref{eq:micond}) for
  $\mcF_x$ in $M_\beta \contract x$ holds because $|E|-1 \ge 0$.

  Finally, the case that $\mcF=\{A\cup B\cup D\cup E, \,\, C\cup A\cup
  D, \,\, F\cup D\cup B\}$ is similar to the previous case, with $D$
  playing the role of $A$.
\end{proof}

\begin{theorem}\label{thm:sboeminor}
  For $k \ge 2$, the matroid $M_\beta$ defined above has the following
  properties:
  \begin{enumerate}
  \item $M_\beta$ is neither $k$-base-orderable nor strongly
    base-orderable,
  \item every single-element contraction of $M_\beta$ is transversal,
  \item $M_\beta$ is an excluded minor for $\mathcal{SBO}$ and
    $k$-$\mathcal{BO}$,
  \item $M_\beta$ is $(k-1)$-base-orderable, and
  \item if $|A| = |B| = |D| = |E| = k/2$, then $M_\beta$ has exactly
    two pairs of bases that have no $k$-exchange-ordering; otherwise,
    there is only one such pair of bases.
  \end{enumerate}
\end{theorem}

\begin{proof}
  We first show that $M_\beta$ is not $k$-base-orderable, and so not
  strongly base-orderable.  As noted above, both $A \cup B \cup C$ and
  $D \cup E \cup F$ are bases of $M_\beta$.  Assume for a
  contradiction that $\sigma \from D \cup E \cup F \to A \cup B \cup
  C$ is a $k$-exchange-ordering.  Since $C\cup B$ is a basis of $C\cup
  B\cup E$, we get $\sigma(E) \subseteq C \cup B$.  Likewise, the fact
  that $C\cup A$ is a basis of $C \cup A \cup D$ forces $\sigma(D)
  \subseteq C \cup A$.  However, it cannot be that $\sigma(D \cup E) =
  C$ since $A\cup B\cup D\cup E$ is a dependent set.  So either
  $\sigma(D) \cap A \ne \varnothing$ or $\sigma(E) \cap B \ne
  \varnothing$.  The former cannot occur because $F\cup E$ is a basis
  of $F\cup E\cup A$, but the latter is also impossible because $F\cup
  D$ is a basis of $F\cup D\cup B$.  Therefore, we have reached a
  contradiction.

  Assertion (2) is Lemma~\ref{lem:sboeminor}.  Since transversal
  matroids are strongly base-orderable, Lemma~\ref{lem:kbodorc}
  implies assertion (3).

  Next we prove assertion (4).  Since proper minors of $M_\beta$ are
  in $\mathcal{SBO}$, we need only show that there is a
  $(k-1)$-exchange-ordering between every disjoint pair of bases.

  Note that nothing distinguishes elements that are in the same set in
  $\beta$. We will use the following consequence of that observation:
  for distinct members $x,y$ of the same set in $\beta$, if $B_1$ is a
  basis of $M_\beta$ with $x \in B_1$ and $y \not\in B_1$, then $(B_1
  - \sg{x}) \cup \sg{y}$ is also a basis of $M_\beta$.

  Let $B_1$ and $B_2$ be disjoint bases of $M_\beta$.  First assume
  that some set $X$ in $\beta$ contains some element, say $x$, in
  $B_1$ and some element, say $y$, in $B_2$.  Now $y \not\in B_1$.  By
  the observation in the previous paragraph, $(B_1 - \sg{x}) \cup
  \sg{y}$ is a basis of $M_\beta$.  Since $M_\beta \contract y$ is
  strongly base-orderable, there is a $k$-exchange-ordering $\sigma
  \from (B_1 - \sg{x}) \cup \sg{y} \to B_2$ with respect to $M_\beta$.
  It must be that $\sigma(y) = y$.  Defining $\tau \from B_1 \to B_2$
  by
  \begin{equation*}
    \tau(e) =
    \begin{cases}
      \sigma(e) & \text{if} \,\, e \ne x \\
      y & \text{if} \,\, e = x,\\
    \end{cases}
  \end{equation*}
  gives a $k$-exchange-ordering with respect to $M_\beta$.

  Now assume that no set in $\beta$ has elements in both $B_1$ and
  $B_2$.  Thus, $B_1$ is a union of sets in $\beta$, as is $B_2$.  The
  size constraints on the sets imply that $B_1$ is a union of at least
  two sets, as is $B_2$.  The only union of two sets that is a basis
  is $C \cup F$, but its complement, $A \cup B \cup D \cup E$, is not
  a basis, so $B_1$ is a union of three sets, as is $B_2$.  Since
  bases have $2k$ elements and $|C|=k$, the fact that $C \cup B \cup
  E$, $C \cup A \cup D$, $F \cup E \cup A$, and $F \cup D \cup B$ are
  cyclic flats implies that there are at most two pairs of disjoint
  bases.  Namely, one of $B_1$ and $B_2$ must be either $A\cup B\cup
  C$ or $C\cup D\cup E$.  If $B_1 = A \cup B \cup C$, then $B_2 = D
  \cup E \cup F$, and any bijection $\sigma \from B_1 \to B_2$ with
  $\sigma(A \cup B) = F$ and $\sigma(C) = D \cup E$ is a
  $(k-1)$-exchange-ordering.  If $B_1 = C \cup D \cup E$, then $B_2 =
  A \cup B \cup F$, and any bijection $\tau \from B_1 \to B_2$ with
  $\tau(C)=A \cup B$ and $\tau(D \cup E)=F$ is a
  $(k-1)$-exchange-ordering.  This proves assertion (4).

  As we now show, it is possible for both $A \cup B \cup F$ and $C
  \cup D \cup E$ to be bases of $M_\beta$ only if $|A| = |B| = |D| =
  |E| = k/2$.  If this equality does not hold, then the larger of
  $|A|$ and $|B|$ must be strictly greater than the smaller of $|D|$
  and $|E|$.  By symmetry, we may assume that $|B| \ge |A|$.  If $|B|
  > |D|$, then since $r(F \cup D \cup B) = |F| + |D|$, it follows that
  $F \cup B$ is dependent.  If instead $|B| > |E|$, then we have $|D|
  > |A|$ since $|A| + |B| = |D| + |E|$.  Now since $r(C \cup A \cup D)
  = |C| + |A|$, it follows that $C \cup D$ is dependent.

  Furthermore, if $A\cup B\cup F$ and $C\cup D\cup E$ are indeed both
  bases of $M_\beta$, then the basis-exchange digraphs $\bed_{A\cup
    B\cup C,\, D\cup E\cup F}$ and $\bed_{C\cup D\cup E,\, A\cup B\cup
    F}$ are isomorphic, and, by symmetry, there is no
  $k$-exchange-ordering between $A\cup B\cup F$ and $C\cup D\cup E$.
  This proves assertion (5).
\end{proof}

Note that, for a given $k\geq 2$, the number matroids $M_\beta$, up to
isomorphism, is the number of $4$-cycles $(p,q,r,s)$ of positive
integers (allowing repetitions) with $p+r=k$ and $q+s=k$.  (We get
$p$, $q$, $r$, $s$ from $|A|$, $|D|$, $|B|$, $|E|$ by some cyclic
shift.)  First let $k=2h+1$, so $k$ is odd.  There are $h$ choices for
the smaller of $p$ and $r$, and likewise for $q$ and $s$.  These two
smallest integers must be adjacent in the cycle.  If the two smallest
integers differ, then the cycle is determined by deciding which
follows the other, so there are $h(h-1)$ such cycles.  If the two
smallest integers are equal, then each of the $h$ choices of that
integer yields only one cycle.  Thus, there are $h(h-1)+h = h^2$
matroids $M_\beta$, up to isomorphism, when $k=2h+1$.  Now let $k=2h$,
so $k$ is even.  The analysis above applies if the two smallest
integers are less than $h$, so there are $(h-1)^2$ such cycles.  If
either of the two smallest integers is $h$, then the cycle is
determined by the other smallest integer, so there are $h$ such
cycles.  Thus, there are $(h-1)^2+h$ matroids $M_\beta$, up to
isomorphism, when $k=2h$.

By Proposition \ref{prop:kbobasics}, the dual matroid $M_\beta^{*}$ is
also an excluded minor for $k$-$\mathcal{BO}$ and $\mathcal{SBO}$ that
is $(k-1)$-base-orderable.  We note that $M_\beta^{*}$ may be thought
of as a variation on $M_\beta$ in the following sense.  Given $\beta =
(A, B, C, D, E, F)$, modify the construction of $M_\beta$ by replacing
the circuit-hyperplane $A\cup B\cup D\cup E$ with the
circuit-hyperplane $C\cup F$, giving the matroid $M'_\beta$, say.  Now
let $\beta' = (A, B, C, E, D, F)$.  One can show that $M_\beta^{*} =
M'_\beta$ by using equation~(\ref{eq:cfrank}) and
Proposition~\ref{prop:cfdual}.

We showed that the single-element contractions of $M_\beta$ are
transversal and hence gammoids.  Perhaps the single-element deletions
of $M_\beta$ are also gammoids, but showing that would require a
different type of argument.  To see why, for $k = 5$, let $|A| = |D| =
2$, $|B| = |E| = 3$, and $|C| = |F| = 5$.  Let $c \in C$.  Using the
Mason-Ingleton condition, one can check that neither $M_\beta \delete
c$ nor $(M_\beta \delete c)^*$ is transversal.  Testing whether
$M_\beta \delete c$ is a gammoid therefore would require a different
approach.


\end{document}